\documentclass{amsart}
\usepackage{graphicx,amsfonts,amssymb,amsmath,amsthm,url,
  verbatim,amscd}
  \usepackage{color}
\usepackage[usenames,dvipsnames]{xcolor}
\usepackage[normalem]{ulem}
\usepackage{enumerate}
\usepackage{pdfsync}
\usepackage{booktabs}
\usepackage[hyperfootnotes=false, colorlinks, citecolor=RoyalBlue,
urlcolor=blue, linkcolor=blue ]{hyperref}

\theoremstyle{plain}
\newtheorem*{theorem*}{Theorem}
\newtheorem{theorem}  {Theorem}    [section]
\newtheorem{lemma}      [theorem]{Lemma}
\newtheorem{corollary}  [theorem]{Corollary}
\newtheorem{proposition}[theorem]{Proposition}

\newtheorem{remark}  [theorem] {Remark}
\theoremstyle{definition}

\newtheorem{definition} [theorem]{Definition}

\newcommand{\as}[1]{{\color{red} \sf AS: [#1]}}

\newcommand{\pn}[1]{{\color{ForestGreen} \sf PN: [#1]}}

\allowdisplaybreaks

\renewcommand{\H}{\mathbb H}
\newcommand{\A}{{\mathbb A}}
\renewcommand{\a}{{\mathbf a}}
\newcommand{\f}{{\mathbf f}}
\renewcommand{\c}{{\mathbf c}}
\newcommand{\F}{{\mathcal F}}
\newcommand{\J}{{\mathcal J}}
\newcommand{\Q}{{\mathbb Q}}
\newcommand{\Z}{{\mathbb Z}}
\newcommand{\N}{{\mathbb N}}

\newcommand{\R}{{\mathbb R}}
\newcommand{\C}{{\mathbb C}}
\newcommand{\bs}{\backslash}

\newcommand{\p}{\mathfrak p}

\newcommand{\OF}{{\mathfrak o}}
\newcommand{\GL}{{\rm GL}}
\newcommand{\PGL}{{\rm PGL}}
\newcommand{\SL}{{\rm SL}}
\newcommand{\SO}{{\rm SO}}

\newcommand{\sgn}{{\rm sgn}}

\newcommand{\sym}{{\rm sym}}

\newcommand{\ur}{{\rm ur}}

\newcommand{\ad}{{\rm ad}}

\newcommand{\Ind}{{\rm Ind}}

\newcommand{\Tr}{{\rm tr}}
\newcommand{\Ad}{{\rm Ad}}

\newcommand{\mat}[4]{{\setlength{\arraycolsep}{0.5mm}\left(
      \begin{array}{cc}#1&#2\\#3&#4\end{array}\right)}}

\def\onehalf{{1/2}}

\def\JL{\operatorname{JL}}

\def\SL{\operatorname{SL}}
\def\vol{\operatorname{vol}}
\def\GL{\operatorname{GL}}

\renewcommand{\Im}{\mathrm{Im}}
\def\eps{\epsilon}

\begin{document}

\bibliographystyle{plain}

\title[Automorphic forms of minimal type]{Some analytic aspects of automorphic forms on $\GL(2)$ of minimal
  type}

\author{Yueke Hu}
\address{Department of Mathematics\\
  ETH Zurich\\
  Raemistrasse 101\\ 8092 Zurich\\
  Switzerland}
\email{huyueke2012@gmail.com}

\author{Paul D. Nelson}
\address{Department of Mathematics\\
  ETH Zurich\\
Raemistrasse 101\\ 8092 Zurich\\
  Switzerland}
\email{paul.nelson@math.ethz.ch}

\author{Abhishek Saha}
\address{School of Mathematical Sciences\\
  Queen Mary University of London\\
  London E1 4NS\\
  UK}
\email{abhishek.saha@qmul.ac.uk}

\begin{abstract}
  Let $\pi $ be a cuspidal automorphic representation of
  $\PGL_2(\A_\Q)$ of arithmetic conductor $C$ and archimedean
  parameter $T$, and let $\phi$ be an $L^2$-normalized
  automorphic form in the space of $\pi$.  The sup-norm problem
  asks for bounds on $\| \phi \|_\infty$ in terms of $C$ and
  $T$.  The quantum unique ergodicity (QUE) problem concerns the
  limiting behavior of the $L^2$-mass $|\phi|^2 (g) \, d g$ of
  $\phi$.  All previous work on these problems in the
  conductor-aspect has focused on the case
  that $\phi$ is a
  newform.

  In this work, we study these problems for a class of
  automorphic forms that are not newforms.
  Precisely, we assume that for each prime divisor $p$ of $C$,
  the local component $\pi_p$ is supercuspidal (and satisfies
  some additional technical hypotheses), and consider
  automorphic forms $\phi$ for which the local components
  $\phi_p \in \pi_p$ are ``minimal'' vectors.
  Such vectors may be understood as non-archimedean analogues of
  lowest weight vectors in holomorphic discrete series
  representations of $\PGL_2(\mathbb{R})$.


  For automorphic forms as above, we prove a sup-norm bound that
  is sharper than what is known in the newform case.  In
  particular, if $\pi_\infty$ is a holomorphic discrete series
  of lowest weight $k$, we obtain the optimal bound
  $C^{1/8 -\eps} k^{1/4 - \eps} \ll_{\eps} |\phi|_\infty
  \ll_{\eps} C^{1/8 + \eps} k^{1/4+\eps}$.
  We prove also that these forms give analytic test vectors for the QUE
  period, thereby demonstrating the equivalence between the strong QUE
  and the subconvexity problems for this class of vectors.
  This finding contrasts the known failure
  of this equivalence
  \cite{NPS} for newforms of
  powerful level.
\end{abstract}

\maketitle

\section{Introduction}\label{s:introduction}
\subsection{Overview}
Let $\pi$ be a cuspidal automorphic representation of
$\GL_2(\A_\Q)$.  Many problems in the analytic number theory of
$\pi$ depend upon the choice of a specific $L^2$-normalized
automorphic form $\phi$ in the space of $\pi$.  For example, the
sup norm, $L^p$-norm and quantum unique ergodicity (QUE)
problems have this feature, while the subconvexity problem does
not.  In such problems, it is customary to work with
factorizable vectors $\phi = \otimes \phi_v$ for which
\begin{equation}\label{eq:customary-choice}
  \phi_\infty = \text{lowest nonnegative weight vector in } \pi_\infty,
  \quad
  \phi_p = \text{newvector in } \pi_p.
\end{equation}
But other reasonable choices are often possible, useful, and
more natural.

A basic illustration of this principle is given by
Lindenstrauss's proof of the QUE theorem.  One formulation of
that theorem is that as $\pi$ traverses a sequence as above for
which $\pi_\infty$ belongs to the principal series, the
$L^2$-masses of the vectors $\phi$ given by
\eqref{eq:customary-choice} equidistribute.  A key step in the
proof is to replace $\phi_\infty$ by another vector
$\widetilde{\phi_\infty}$ (the microlocal lift) whose limit
measures acquire additional invariance.  Further illustration of
this principle is given by period-based approaches to the
subconvexity and shifted convolution problems (see e.g. \cite{MR2726097,michel-2009,MR2437682}).

This work explores a particular choice for the local components
$\phi_p$ which turn out to have several remarkable properties.
Briefly, assuming that $\pi_p$ is supercuspidal and that its
conductor is a fourth power, we consider $\phi_p$ which are
analogues of the lowest weight vectors in holomorphic discrete
series representations of $\PGL_2(\mathbb{R})$; see Section
\ref{s:introdefs} for a more detailed description of these
vectors from this point of view and Definition \ref{def:waldsopt} for the formal definition.
We aim to demonstrate the strength of our analogy from the
analytic perspective by illustrating with two examples: the
sup norm problem and the QUE problem.

For lack of better
terminology, we refer to these vectors as \emph{minimal vectors}
or \emph{vectors of minimal type}.  (When $\pi_p$ belongs to the
principal series, analogous vectors were studied in
\cite{nelson-padic-que}.)  Minimal vectors are implicit in the
type theory approach to the construction of supercuspidal
representations, as in the works of Howe \cite{MR0492088,
  MR0492087}, Kutzko \cite{MR0507254}, Moy \cite{MR987030}, Bushnell \cite{MR882297}, and others. On the other hand, their analytic properties, in
the sense of the problems recalled above, do not appear to have
been explored.  The purpose of the present work is to fill this
gap.


Before describing in detail the vectors to be studied,
we indicate some of the intended applications.

\subsection{The sup norm problem in the level aspect}
Assume that $\phi = \otimes \phi_v$, with $\phi_\infty$ a vector of lowest non-negative weight and $\phi_p$ spherical for all primes $p \nmid C$.  Then $\phi$ corresponds to either
a Hecke--Maass cusp form $f$ of weight $k\in \{0,1\}$ and
Laplace eigenvalue $\lambda$ or to a holomorphic Hecke
eigencuspform $f$ of weight $k \in \Z_{>0}$ (with respect to some congruence subgroup).  The $\GL(2)$ sup-norm
problem asks for bounds on
$\| \phi \|_\infty = \|y^{k/2}f\|_\infty$ in terms of $C$ and
$k/\lambda$ and has been much studied recently.
(A variant of this problem asks for bounds on
$\|f|_{\Omega}\|_\infty$, where $\Omega$ is a fixed compact
set. This formulation avoids the cusps and focusses on behavior
at the bulk. We do not discuss this variant in the present
paper.)

In the case $C=1$ and $f$ a
Hecke--Maass cusp form of weight 0 for $\SL_2(\Z)$,  Iwaniec and
Sarnak \cite{iwan-sar-85} proved the pioneering result
$\lambda^{1/12-\eps} \ll_\eps \|f\|_\infty \ll_\eps
\lambda^{5/24 + \eps}$.
Their proof combined the
Fourier expansion with a subtle amplification argument. On the
other hand, for $f$ a holomorphic cuspidal eigenform of weight
$k$ for $\SL_2(\Z)$, the Fourier expansion alone turns out to be
sufficient to get the optimal exponent in the weight aspect;
this was worked out by Xia \cite{xia}, who
proved \begin{equation}\label{e:xia}k^{1/4-\eps} \ll_\eps
  \|y^{k/2}f\|_\infty \ll_\eps k^{1/4 + \eps}.\end{equation}

For $C>1$, one needs to make a choice for $\phi_p$ at each prime
$p$ dividing $C$.
The customary choice has been
to take the newvector at each prime. The corresponding forms $f$
are (Hecke--Maass or holomorphic) \emph{newforms} with respect
to the group $\Gamma_1(C)$. For such newforms and for \emph{squarefree}
$C$ there were several results \cite{blomer-holowinsky,
  harcos-templier-1, harcos-templier-2, templier-sup,
  templier-sup-2} culiminating in the bound
$\|\phi\|_\infty \ll_{k/\lambda,\eps} C^{1/3 + \eps}$ due to
Harcos and Templier.
(Here, for simplicity, we have quoted the bound
only in the conductor-aspect,
noting that a hybrid result was proved by
Templier in \cite{templier-sup-2}.)
This bound was generalized
to the case of powerful (non-squarefree) $C$ by the third
author \cite{saha-sup-level-hybrid}. In the special case of
trivial central character, and again focussing only on the
conductor aspect, the results of \cite{saha-sup-level-hybrid}
give \begin{equation}\label{saha:result}\|\phi\|_\infty
  \ll_{k/\lambda,\eps} C_0^{1/6 + \eps}
  C_1^{1/3+\eps},\end{equation} where we write $C=C_0C_1$ with
$C_0$ the largest integer such that $C_0^2$ divides $C_1$. Note
that $C_0^{1/6} C_1^{1/3}$ always lies between $C^{1/4}$ and
$C^{1/3}$.

The above bounds have been recently extended to the case of
newforms over number fields, initially covering only squarefree
conductor \cite{blomer-harcos-milicevic,
  blomer-harcos-milicevic-maga} and more recently, for all
conductors by Edgar Assing (to appear in his forthcoming Bristol
thesis). As for lower bounds, one only knows the trivial bound
$1 \ll \|\phi\|_\infty$ in general; however in the case when the
central character is highly ramified, there exist results giving
large lower bounds \cite{sahasupwhittaker,templier-large} due to
the unusual behavior of local Whittaker newforms (the
corresponding best-known upper bounds are also worse in these
cases).

Thus, the state-of-the-art for the $\GL(2)$ sup-norm problem may
seem quite satisfying. Nonetheless there is a key deficiency in
all the works so far --- they focus exclusively on newforms. The
situation for Hecke eigenforms that correspond at the ramified
places to interesting local vectors that are not newvectors
remains completely unexplored.  One aim of this paper is to
explore the sup norm problem when $\phi_p$ is a minimal vector at each prime $p$ dividing $C$.
As indicated above, these local vectors may be viewed as
$p$-adic analogues of holomorphic vectors at infinity.  The
corresponding global automorphic forms $\phi$ will be referred
to as automorphic forms of minimal type. For such forms, we prove a level aspect sup-norm bound that is strongly analogous to the weight aspect bound \eqref{e:xia}.

\begin{theorem}[See Theorem \ref{t:globalmain} for a more general hybrid version]\label{t:mainintro}Let $\pi \simeq \otimes_v \pi_v$ be an irreducible, unitary, cuspidal automorphic
  representation of $\GL_2(\A)$ with trivial central character
  and conductor $C$. Assume that $C=N^4$ is the fourth power of
  an odd integer $N$ and suppose, for each prime $p$ dividing
  $C$, that $\pi_p$ is a supercuspidal representation. Let
  $\phi$ be an $L^2$-normalized automorphic form in the space of
  $\pi$ that is of minimal type. Then
$$C^{\frac18 - \eps} \ll_{k/\lambda,\eps} \|\phi \|_\infty \ll_{k/\lambda, \eps} C^{\frac18+ \min(\frac{1}{32}, \frac{\delta_\pi}{2})+ \eps}.$$ Above, $\delta_\pi$ is any exponent towards the Ramanujan conjecture for $\pi$; in particular we may take $\delta_\pi = 0$ if $\pi_\infty$ is holomorphic and $\delta_\pi =7/64$ otherwise.
\end{theorem}

The upper-bound in Theorem \ref{t:mainintro} is much stronger than what is known when $\phi$ is a newform
(with the same assumptions on $\pi$ as above). In the newform case, the best known
upper bound \cite{saha-sup-level-hybrid} is $C^{1/4+ \eps}$,
which is just the ``local bound" in the level aspect (both for newforms as well as for the minimal automorphic forms considered here).  The bound obtained in this paper gives the first instance of an automorphic form of powerful level for which the local sup-norm bound in the level aspect has been improved upon. Furthermore, our bound is \emph{optimal} in the case when $\phi$ corresponds to a holomorphic cusp form, and the proof (as we will see) relies only on
the Whittaker/Fourier expansion. Thus, it is very close to Xia's
result \cite{xia} in many respects.
\subsection{Period integrals for QUE}\label{sec:intro-period-integrals-que}
Going back to the holomorphic newform case, assume that the local components of $\phi$ are given by
\eqref{eq:customary-choice}, that $\pi$ has trivial central character, and that $\pi_\infty$ is a holomorphic discrete series of lowest weight $k$.  Then $\phi$ corresponds to a holomorphic newform $f$ of weight $k$ with respect to $\Gamma_0(C)$.  For each Hecke--Maass cusp form $g$ of weight 0 for $\SL_2(\Z)$, define $$D_f(g) =  \frac{\int _{\Gamma_0(C) \backslash \mathbb{H} }
y ^k \lvert f \rvert ^2 (z)  g(z) \, \frac{d x \, d y }{ y ^2 }
}{\int _{\Gamma_0(C) \backslash \mathbb{H} }
y ^k \lvert f \rvert ^2 (z)   \frac{d x \, d y }{ y ^2 }}.$$

The problem of proving $D_f(g)\rightarrow 0$
for fixed $g$
as the parameters $C$ and $k$ of $f$ grow
is a natural analogue of the Rudnick--Sarnak
quantum unique ergodicity (QUE) conjecture \cite{MR1266075}. It was proved by
Holowinsky and Soundararajan \cite{holowinsky-soundararajan-2008} that
$D_f(g) \rightarrow 0$ for fixed $C$ ($=1$) and varying $k
\rightarrow \infty$; we refer to their paper
and \cite{sarnak-progress-que}
for further historical background.
The case of varying squarefree levels
was addressed in \cite{PDN-HQUE-LEVEL},
where it was shown that
$D_f(g) \rightarrow 0$ as $Ck \rightarrow \infty$
provided that $C$ is squarefree. Finally, it was proved in \cite{NPS} that $D_f(g) \rightarrow 0$
  whenever $C k \rightarrow \infty$ (without any restriction on $C$). In fact, the main result of \cite{NPS} gave an unconditional power savings bound $D_f(g) \ll_g C_0^{-\delta_1} \log(Ck)^{-\delta_2}$ for some positive constants $\delta_1, \delta_2$, where as before, $C_0$ denotes the largest integer such that $C_0^2|C$. Further extensions of this result to the case when $g$ is not of full level were obtained in \cite{Hu:17a}.

  There is a marked difference above between the case when $C$ is squarefree and the case when $C$ is powerful. For $C$ squarefree,
a generalization
of Watson's formula
(see \cite{PDN-HQUE-LEVEL} for a precise version)
asserts
that for each
$g$ as above, corresponding to an automorphic representation $\sigma_g$,
one has
\begin{equation}\label{eq:32}
\left\lvert D_{f}(g) \right\rvert^2
= (C k)^{-1+o(1)}
L(\pi \times \pi \times \sigma_g,\onehalf).
\end{equation}
Here the convexity bound reads $L(\pi \times \pi \times \sigma_g,\onehalf) \ll (Ck)^{1+o(1)}$. Thus, for squarefree levels, the subconvexity and QUE problems are essentially equivalent. A major point of \cite{NPS} was that this equivalence is no longer true for powerful levels. For example, in the case when $C$ is a perfect square, the results of \cite{NPS} imply that $
  \left\lvert D_f(g)  \right\rvert^2
  \ll_{g,k} C^{\theta-1}
L(\pi \times \pi \times \sigma_g,\onehalf)
$
where $\theta = 7/64$. The convexity bound in this case gives $L(\pi \times \pi \times \sigma_g,\onehalf) \ll_k C^{1/2+o(1)}$. So in this case, the convexity bound alone is enough to imply QUE with power savings in the level aspect! More generally, as shown in \cite{NPS}, the QUE problem is significantly \emph{easier} than the subconvexity problem in the case of newforms of powerful level (in contrast to the squarefree case, where these problems are essentially equivalent).

One may ask whether the equivalence between QUE and subconvexity
might be recovered for powerful levels by replacing the newform
with a different choice of vector. We show that this is indeed
the case for automorphic forms having a local component of
minimal type in a supercuspidal
representation
of fourth power conductor.
For a related observation when the local component
belongs to a principal series representation,
see \cite[Rmk 30]{nelson-padic-que}.

Let $\pi$, $C=N^4$ and $\phi$ be as in Theorem
\ref{t:mainintro}. We assume that $\pi_\infty$ is a holomorphic
discrete series of lowest weight $k$. We can associate to $\phi$
a holomorphic modular form $f$ defined by $f(z) = j(g_\infty,
i)^k\phi(g_\infty)$ where $g_\infty \in \SL_2(\R)$ is any matrix
such that $g_\infty i = z$. We let $\Gamma$ denote any
congruence subgroup such that $f|_k \gamma = f$ for all $\gamma
\in \Gamma$ (we will see later that we may take $\Gamma =
\Gamma(N^2)$). We prove the following result.

\begin{theorem}\label{thm:watson-ext}
  Let $g$ be a Hecke-Maass cuspform for $\SL_2(\Z)$, and let $\sigma_g$ be the automorphic representation generated by (the adelization of) $g$.
  Then
  $$\frac{
      \left\lvert \int_{\Gamma \backslash \mathbb{H}}
        g(z) |f|^2(z) y^k \, \frac{d x \, d y}{y^2}
      \right\rvert^2
    }{
      \left(\int_{\SL_2(\Z) \backslash \mathbb{H}}
        |g|^2(z)  \, \frac{d x \, d y}{y^2} \right)
      \left( \int_{\Gamma \backslash \mathbb{H}}
        |f|^2(z) y^k \, \frac{d x \, d y}{y^2}
      \right)^2
    }
    =  \frac{1}{8 }
    \frac{\Lambda(\pi \times \pi \times \sigma_g,\onehalf)}{
      \Lambda(\ad \sigma_g,1) \Lambda(\ad \pi,1)^2} \prod_{p| C}  I_p,
    $$ where each local factor $I_p$ satisfies
    $$I_p \asymp \rm{Cond}(\pi_\p\times \pi_p)^{-1/2}.$$
\end{theorem}
In the above case, the convexity bound reads $\Lambda(\pi \times \pi \times \sigma_g,\onehalf) \ll_{g,k} C^{1/2 + o(1)} = \rm{Cond}(\pi
\times \pi)^{1/2 + o(1)}$. So Theorem \ref{thm:watson-ext} shows that for the family of cusp forms coming from minimal vectors, the QUE and subconvexity problems are essentially equivalent. In fact, our local results imply more general identities in which
$g$ is allowed to have some level.

It is very likely that, by combining Theorem
\ref{thm:watson-ext} with the arguments of \cite[Sec. 3]{NPS}, one
could establish the estimate $D_f(g) \ll \log (C k)^{-\delta}$
for small $\delta > 0$ and fixed $g$, but we do not pursue this
here.

\subsection{Automorphic forms of minimal type}\label{s:introdefs}
In the rest of this introduction, we explain in detail the
concept of an automorphic form of minimal type and touch upon
some of the key ideas that power our results.

Let $\pi \simeq \otimes_v \pi_v$ be an irreducible, unitary, cuspidal representation of $\GL_2(\A_\Q)$ of conductor $C$. We assume henceforth for simplicity that the central character of $\pi$ is trivial. An automorphic form $\phi=\otimes_v \phi_v$ in the space of $\pi$ can be constructed out of any choice of local vectors $\phi_v\in \pi_v$ such that $\phi_p$ is spherical ($\GL_2(\Z_p)$-fixed)  at almost all primes $p$. It is very natural to choose $\phi_p$ to be the (unique up to multiples) spherical vector   at \emph{all} primes not dividing the conductor $C$, and we will always do so. At the archimedean place, we will choose $\phi_\infty$ to be a vector of minimal non-negative weight $k$, i.e., with the property \begin{equation}\label{weighteqintro}\pi_\infty \mat{\cos(\theta)}{\sin(\theta)}{-\sin(\theta)}{\cos(\theta)} \phi_\infty = e^{ik \theta}\phi_\infty\end{equation} where $k$ is the smallest non-negative integer (which in our case must be an even integer as the central character is trivial) for which the above equality holds for some $\phi_\infty$. Note that $k=0$ if $\pi_\infty$ is a principal series representation and $k \ge 2$ if $\pi_\infty$ is a discrete series representation.

Now, consider the primes $p$ dividing $C$. What should we take $\phi_p$ to be? One standard possibility would be to take $\phi_p$ to be the newvector, i.e., fixed by a congruence subgroup of the form $\begin{bmatrix}
  1+ p^c \Z_p  & \Z_p  \\
  p^c \Z_p  & \Z_p^\times
\end{bmatrix}$ where $c$ is taken as small as possible, whence
newform theory implies $c= v_p(C)$.

The minimal vectors studied in this paper may be viewed as an
alternative to the newvector in many cases.
  As we now explain, they may be regarded as
  non-archimedean analogues of the holomorphic (lowest weight) vector at infinity for a discrete series.
Let $T_\infty:=\R^\times \SO(2)$ be the standard maximal non-split torus inside $\GL_2(\R)$; we have the isomorphism $T_\infty \cong \C^\times$ sending $r \mat{\cos(t)}{\sin(t)}{-\sin(t)}{\cos(t)}$ to $re^{it}$. Let $\theta_{\pi_\infty}$ be the  character on $\C^\times$ given by $\theta_{\pi_\infty}: r e^{i t} \mapsto e^{ikt}$ which we may view as a character on $T_\infty$. Then the equality \eqref{weighteqintro} may be restated as  \begin{equation}\label{e:weightintronew}\pi_\infty(t_\infty) \phi_\infty = \theta_{\pi_\infty}(t_\infty)\phi_\infty, \quad t_\infty \in T_\infty.\end{equation}
The character $\theta_{\pi_\infty}$ depends only on $k$ and is therefore an invariant attached to $\pi_\infty$.

Let us further explicate the relation between $\pi_\infty$ and $\theta_{\pi_\infty}$ when  $\pi_\infty$ is a discrete series representation. Let $\xi_{\pi_\infty}$ be the character on $\C^\times$ given by $re^{it} \mapsto e^{i(k-1)t}$. By a special case of the local Langlands correspondence --- see \cite[(3.4)]{knapp} and note that $\pi_\infty \simeq D_{k-1}$ in the notation of \cite{knapp} ---  the  $L$-parameter of $\pi_\infty$  under the local Langlands correspondence is the representation $\Ind_{W_\C}^{W_\R} \xi_{\pi_\infty}$ of the real Weil group $W_\R$; equivalently, the representation $\pi_\infty$ is obtained by \emph{automorphic induction} from the character $\xi_{\pi_\infty}$ of $\C^\times$. Let $\eta_\C$ be the character  on $\C^\times$ given by  $r e^{i t} \mapsto e^{it}$ which we may think of as the simplest extension of the sign character on $\R^\times$ to $\C^\times$. Then we have $\theta_{\pi_\infty} = \eta_\C \xi_{\pi_\infty}$.

Next, take $p$ to be a prime dividing $C$. Then there is a unique unramified quadratic field extension $E_p$ of $\Q_p$ which should replace $\C$ in our analogy.  As in the archimedean case, we can specify a maximal non-split torus $T_p$ inside $\GL_2(\Z_p)$ such that $T_p \simeq E_p^\times$; without loss of generality we may assume that $T_p$ is in canonical form (see Definition \ref{d:inertorus}). Now, suppose that $\pi_p$ is a supercuspidal representation of even minimal (exponent of) conductor. Then, similarly to above, $\pi_p$ is obtained by automorphic induction from some regular character $\xi_{\pi_p}$ of $E_p^\times$ (see  \cite[Prop. 3.5]{tunnell78}). Let $\eta_{E_p}$ be the unique unramified extension to $E_p^\times$ of the quadratic character on $\Q_p^\times$ associated to the extension $E_p/\Q_p$ by local class field theory. We view $\eta_{E_p}$ as the non-archimedean analogue of the character $\eta_\C$ defined earlier. Define the character $\theta_{\pi_p}$ on  $T_p \simeq E_p^\times$ by  $\theta_{\pi_p} = \eta_{E_p} \xi_{\pi_p}$, which  is then the analogue of the character $\theta_{\pi_\infty}$ on $T_\infty \simeq\C^\times$ defined above. Analogously to \eqref{e:weightintronew}, we define a \emph{minimal vector} to be any non-zero element $\phi_p$ in the space of $\pi_p$ such that \begin{equation}\label{e:waldseqintro}\pi_p(t_p) \phi_p = \theta_{\pi_p}(t_p)\phi_p, \quad t_p \in T_p.\end{equation}

The comparison of \eqref{e:weightintronew} and
\eqref{e:waldseqintro} shows that minimal vectors are
the non-archimedean analogues of the lowest weight
(holomorphic) vectors in archimedean discrete series
representations. The minimal vectors also occur naturally from
the point of view of microlocal analysis, in that they are
approximate eigenvectors under the action by small elements of
the group;
they are in this sense analogous also to the $p$-adic microlocal
lifts
considered in \cite{nelson-padic-que}.
We
remark here that given a character $\chi_p$ of $T_p\simeq
E_p^\times$, a $T_p$-eigenvector with eigencharacter $\chi_p$ is a vector
$\phi_p$ that satisfies $\pi_p(t_p) \phi_p = \chi_p(t_p)\phi_p$
for each $t_p \in T_p$.
The choice
$\chi_p=\theta_{\pi_p}$ corresponds to our case, whereby  the vector acquires some  remarkable properties.

The Saito--Tunnell theorem \cite{Tunnell83, Saito1993}
implies that a minimal
vector, \emph{if it exists}, is unique up to multiples
(once the group $T_p$ is fixed);
moreover, a
minimal vector exists if and only if
$\eps(1/2, \pi_p \otimes \mathcal{AI}(\theta_{\pi_p}^{-1}))=1$ (where $\mathcal{AI}$ denotes automorphic induction from $E_p^\times$).
We verify in Proposition \ref{p:existencewalds} below
that if $p$ is
odd, $v_p(C)$ is a multiple of 4, and $\pi_p$ is supercuspidal,
then a minimal vector (as we have defined it) indeed exists. Precisely, given such a $\pi_p$, the character $\theta_{\pi_p}$ of $T_p$ can be extended to a character $\chi_{\pi_p}$ of the compact-mod-centre group $L:=T_p(1+p^nM_2(\Z_p))$ (where $n= \frac{v_p(C)}{4}$) with the property that $\pi_p \simeq c-\Ind_{L}^{G}\chi_{\pi_p}.$ The restriction of $\pi_p$ to $L$ contains $\chi_{\pi_p}$, which gives an immediate proof of existence. Incidentally, the  pair $(L,   \chi_{\pi_p})$ is in some sense the smallest possible among all inducing pairs for $\pi_p$ and constitutes a minimal $K$-type in the   sense of Moy \cite{MR987030}. Therefore, a minimal vector, in our setup, is precisely one that generates the (one-dimensional) minimal $K$-type associated to $\pi_p$. This is one of the reasons for our use of the term ``minimal" to describe these vectors.

Returning to the global setup, we suppose that $C=N^4$ is the fourth power of an odd integer, and $\pi_p$ is supercuspidal at all primes dividing $C$. Then, by choosing $\phi_p$ to be a minimal vector at each prime $p$ dividing $C$, we can construct a global automorphic form $\phi = \otimes_v \phi_v$ in the space of $\pi$; we call this an automorphic form of minimal type. It is precisely for such forms $\phi$ that our Theorem \ref{t:mainintro} applies.

We end this subsection with a brief discussion of what an automorphic form $\phi$ of minimal type looks like classically. We can associate to $\phi$ a function $f$ on $\H$ defined by $f(z) = j(g_\infty, i)^k\phi(g_\infty)$ where $g_\infty \in \SL_2(\R)$ is any matrix such that $g_\infty i = z$. Then there exists an integer $D$ and a character $\chi_\pi$ on the ``toric'' congruence group
$$\Gamma_{T,D}(N) := \left\{ \mat{a}{b}{c}{d} \in \SL_2(\Z): a \equiv d \pmod{N}, \ c \equiv -bD \pmod{N} \right \}$$ such that $$f|_k \gamma = \chi_\pi(\gamma)f, \quad \gamma \in \Gamma_{T,D}(N).$$ The character $\chi_\pi$ turns out to be trivial on the principal congruence subgroup of level $N^2$ which is contained in $\Gamma_{T,D}(N)$; see Remark \ref{r:adelic} for more details.   Thus, $f$ is a (very special) member of the space of (holomorphic or Maass) Hecke eigencuspforms of weight $k \in 2\Z$ with respect to the principal congruence subgroup of level $N^2$. Theorem \ref{t:mainintro} gives the optimal sup-norm bound in the conductor aspect (assuming the Ramanujan conjecture) for such forms $f$, just as \eqref{e:xia} gives the optimal sup-norm bound in the weight aspect for holomorphic eigencuspforms. This fits nicely with our analogy between holomorphic vectors at infinity and minimal vectors at a finite prime.
\subsection{The Whittaker expansion}The strong bound in Theorem
\ref{t:mainintro} is obtained purely from the Whittaker
(Fourier) expansion, and depends on an important property of
minimal vectors. We now explain this.

As before, let  $\pi \simeq \otimes_v \pi_v$ be an irreducible, unitary, cuspidal representation of $\GL_2(\A_\Q)$ of conductor $C = N^4=\prod_p p^{4 n_p}$ and of trivial central character. We begin with some general discussion, which applies to any automorphic form $\phi$ in the space of $\pi$.  The Whittaker expansion for $\phi$, which we want to exploit to bound $|\phi(g)|$, looks as follows, $$\phi(g) = \sum_{q \in \mathbb{Q}_{\neq 0}}
W_\phi(\mat{q}{}{}{1} g)$$ where $W_\phi$ is the global Whittaker function attached to $\phi$. Let $g = g_\f g_\infty \in G(\A),$ where $g_\f$ denotes the finite part of $g$ and $g_\infty$ denotes the infinite component. There is an integer $Q(g_\f)$, depending on $g_\f$, such that the Whittaker expansion above is supported only on those $q$ whose denominator divides $Q(g_\f)$. Moreover, the sum decays very quickly after a certain point $|q| > T(g_\infty)$ due to the exponential decay of the Bessel function. The upshot is that \begin{equation}\label{whitintro}\phi(g) = \sum_{\substack{m \in \Z_{\neq 0}}}
W_\phi(\mat{m/Q(g_\f)}{}{}{1} g)\end{equation} with only the
terms $|m|< Q(g_\f)T(g_\infty)$
contributing essentially.

Now, suppose that $\phi$ is an automorphic form of minimal
type.
We let $g_\f$ vary over the set $\prod_{p |C}\GL_2(\Z_p)$ and
$g_\infty$ vary over the set $\mat{y}{x}{}{1}$ with $y \ge
\sqrt{3}/2$. This gives a generating domain, similar to the one
used in \cite{saha-sup-level-hybrid}, and leads to $Q(g_\f) =
N^2$. Using this alone, a standard argument (see the discussion in Section 1.4 of \cite{saha-sup-level-hybrid})
   gives the bound $|\phi(g)| \ll_{k/\lambda, \eps} C^{1/4 +\eps}$, which is  as good as the best known bound in the case of newforms. Incidentally, it turns out that $C^{1/4+\eps}$ is the ``local bound" in our case just as it is in the case of newforms of conductor $C$. This follows from Corollary \ref{cor:matrixcoeff}. Here, we use the term ``local bound" in the sense of \cite{marsh15}.

 Theorem \ref{t:mainintro} of course, goes beyond the local bound, and indeed gives the optimal bound under Ramanujan. What allows us to do this is the following key property of the local Whittaker function $W_{\phi_p}$ associated to a minimal vector, namely, for each $k \in \GL_2(\Z_p)$ there exists some $\a_k \in \Z_p^\times$  such that $W_{\phi_p}(\mat{q}{}{}{1}k) \neq 0$ for $q \in \Q_p^\times$ if and only if $p^{2n_p}q $ belongs to $\Z_p^\times$ and satisfies $p^{2n_p}q  \equiv \a_k \pmod{p^{n_p}}$. In sharp contrast,
the formula for the Whittaker function of a newvector involves a sum of twisted
$\GL_2$-epsilon factors \cite[Section 2.7]{sahasupwhittaker},
with a likely cancellation that seems difficult to prove.

Using the factorization of global Whittaker functions, it follows that \eqref{whitintro} takes the form
\begin{equation}\label{whitintro2}\phi(g) = \sum_{ m \equiv A \bmod{N}}
W_\phi(\mat{m/N^2}{}{}{1} g)\end{equation} for some integer
$A$. In other words, the Whittaker function  of an automorphic
form of minimal type is supported on an
\emph{arithmetic progression}.

This last point can also be explained classically. Suppose that $\pi_\infty$ is a holomorphic discrete series of lowest weight $k$, in which case $\phi$ corresponds to a holomorphic modular form $f$ with respect to the group $\Gamma_{T,D}(N)$. Then the above discussion implies that the Fourier expansion of $f$ at any cusp $\alpha = \sigma(\infty)$ takes the form \begin{equation}\label{fourier3}
(f|_k\sigma)(z) = \sum_{\substack{n>0 \\ n \equiv b \bmod{N}}}a_f(n;\alpha)e^{2 \pi i nz/N^2}.
\end{equation}

It is precisely the fact that the Fourier coefficients above are supported on an arithmetic progression that
allows us to get the additional savings  beyond the local bound.
\subsection{Further remarks}The minimal vectors have
many other important properties that we have not discussed
above. Perhaps their most striking feature is that the matrix
coefficient associated to an $L^2$-normalized
minimal vector is a \emph{character} of the supporting
subgroup (see Proposition \ref{keymatrixprop}).
This matrix coefficient formula can be
easily used to calculate the local integrals of
Gan--Gross--Prasad type involving a minimal vector
(as in the proof of Theorem \ref{thm:watson-ext}).
More generally,
one might hope to use such vectors in classical period formulas
(e.g., in Waldspurger's formula or the triple product formula) with a view
towards applications to subconvexity, mass equidistribution,
$L^p$-norms, arithmetic of special $L$-values, and so on;
Theorem \ref{thm:watson-ext}
may be understood as a first step in that direction.

The fact that the matrix coefficient of a minimal vector turns
out to be a character also has another very interesting
interpretation, which further justifies our use of the word
``minimal.'' By formal degree considerations, the integral of
the square of the matrix coefficient associated to an
$L^2$-normalized vector in a square-integrable local
representation $\pi_p$ of conductor $p^{c_p}$ is
\emph{independent} of the choice of vector, and equals roughly
$p^{-c_p/2}$. The matrix coefficient of an $L^2$-normalized
minimal vector is a character and so has \emph{maximum} possible
absolute value on the support (since the absolute value of a
matrix coefficient of an $L^2$-normalized vector can never
exceed 1, by the triangle inequality). Therefore the minimal
vectors have the property that their matrix coefficients have as
small support as possible!

Incidentally, this last fact makes such a matrix coefficient a
great choice as a test function in the pre-trace formula for
amplification purposes, since small support translates to more
congruences for counting purposes. Indeed, while Theorem
\ref{t:mainintro} does not rely on any sort of amplification,
one could consider the analogous sup-norm problem for automorphic forms of minimal type on a compact quotient of the upper half-plane. In this case, while there is no Whittaker expansion, an amplification argument should allow one to achieve an upper bound for the sup-norm in the conductor aspect that improves upon the local bound. One could also consider analogous problems for quaternion algebras ramified at infinity, where similarly strong bounds may be expected from amplification. We suppress further discussion of this topic in the interest of brevity.

 Next, we say a few words about the restriction to $C$ being a
 fourth power of an odd integer, and $\pi_p$ being supercuspidal
 at all primes dividing $p$. These conditions can in fact be
 removed when $p$ is not equal to 2, provided one is happy to
 slightly relax the definition of minimal vector. To give an
 example, consider the case of an odd prime $p$ such that
 $\pi_p$ is supercuspidal of even minimal (exponent of)
 conductor but $v_p(C) \equiv 2 \pmod{4}$. In this case, no
 vector satisfying \eqref{e:waldseqintro} exists (the
 Saito--Tunnell criterion is not satisfied). However, if one
 were to slightly perturb \eqref{e:waldseqintro} by multiplying
 $\theta_{\pi_p}$ by any character of conductor $p$, then
 vectors satisfying the resulting equality indeed exist. Similar
 discussion
 (roughly in the spirit of \cite{nelson-padic-que})
 applies to principal series representations (one needs to take $E_p = \Q_p \times \Q_p$ in this case) as well as dihedral supercuspidals with odd minimal (exponent of) conductor, for which we should take $E_p$ to be a ramified quadratic extension of $\Q_p$. Indeed, if $p \neq 2$, every case can be covered, leading to a comprehensive theory of such ``almost-minimal'' vectors that takes care of every type of representation. The sup-norms of the resulting automorphic forms of almost-minimal type can be studied similarly, though the bounds will be sometimes slightly worse than what we get.

  The case of $p=2$ is much more subtle due to the presence of non-dihedral supercuspidals, and currently it is not clear to us how to define minimal vectors in that case. One general possibility in every case might be to consider a vector inside a minimal $K$-type. The details of this theory over $\GL(2)$ can be found in \cite{MR0507254}. Such a definition should in fact work not just for $\GL(2)$ but for all reductive groups, using a well-known theorem of Moy--Prasad \cite{moy-prasad} on the existence of unrefined minimal $K$-types for irreducible, admissible representations of $p$-adic reductive groups. It would be very interesting to see if these ideas can be used to study the sup-norm problem in the level aspect for higher rank groups.

\subsection*{Acknowledgements}
Y.H. thanks Simon Marshall for some general discussions about the sup-norm problem. A.S. thanks Emmanuel Kowalski and Ralf Schmidt for helpful
discussions on certain topics related to this paper. The authors thank Valentin Blomer for discussions related to Proposition \ref{p:sym4}, and the anonymous referee for helpful comments which have improved this paper.

Y.H. gratefully acknowledges the support of SNF grant SNF-169247 during the work leading to this paper. Part of this paper is based upon work  supported by the National
Science Foundation under Grant No. DMS-1440140 while Y.H. and P.N. were in residence at the Mathematical Sciences Research Institute in Berkeley, California, during the Spring 2017 semester.
\subsection*{Notations}
We collect here some general notations that will be used throughout this paper. Additional notations will be defined where they first appear in the paper.

Let $\H$ denote the upper half plane and $\GL_2(\R)^+$ the group of real two-by-two matrices with positive determinant. For $z \in \H$, $\left(\begin{smallmatrix}a&b\\c&d
\end{smallmatrix}\right) \in \GL_2(\R)^+$, we let $\left(\begin{smallmatrix}a&b\\c&d
\end{smallmatrix}\right)z=\frac{az+b}{cz+d} \in \H$ be the point obtained by M\"obius transformation.  Given a function $f$ on $\H$, an integer $k$, and some $\gamma = \left(\begin{smallmatrix}a&b\\c&d
\end{smallmatrix}\right) \in \GL_2(\R)^+$, we define a function $f |_k \gamma$ on $\H$ via $(f |_k \gamma)(z) = \det(\gamma)^{k/2} (cz+d)^{-k} f(\gamma z).$

For any two complex numbers $\alpha, z$, we let $K_\alpha(z)$ denote the modified Bessel function of the second kind. The symbol $\varphi$ denotes the Euler totient function.

 For elements $x$, $y$, $t$ in some ring $R$, we define the following two-by-two matrices over $R$: $$a(y) = \begin{bmatrix}
  y &  \\
  & 1
\end{bmatrix},
\quad
n(x) = \begin{bmatrix}
  1 & x \\
  & 1
\end{bmatrix},
\quad z(t)
= \begin{bmatrix}
  t &  \\
  & t
\end{bmatrix}.
$$

 We use the notation
$A \ll_{x,y,\ldots} B$
to signify that there exists
a positive constant $C$, depending at most upon $x,y,\ldots$
so that
$|A| \leq C |B|$. The absence of the subscripts $x,y,\ldots$ will mean that the constant $C$ is universal. We will use $A \asymp B$ to mean that $B \ll A \ll B$.
 The symbol $\epsilon$ will denote a small positive quantity. The values of $\epsilon$ and that of the constant implicit in $\ll_{\epsilon}$ may change from line to line.

We shall always assume every character  is continuous (but not necessarily unitary). The convention used for our Hermitian inner products is that they are linear in the first variable.


\section{Minimal vectors and their Whittaker functions}\label{sec:localcalcs}
This section will be purely local.

\subsection{Preliminaries on fields, characters and representations}\label{sec:2-notations}
Let $F$ denote a non-archimedean local
  field  of characteristic zero. We assume throughout that $F$ has \emph{odd} residue cardinality $q$.
Let $\mathfrak{o}$ be its ring of integers,
and $\mathfrak{p}$ its maximal ideal.
Fix a uniformizer $\varpi$ of $\OF$ (a choice of generator of $\mathfrak{p}$) .
Let $|.|$ denote the absolute value
on $F$ normalized so that
$|\varpi| = q^{-1}$. For each $x \in F^\times$, let $v(x)$ denote the integer such that $|x| = q^{-v(x)}$.  For a non-negative integer $m$, we define the subgroup $U_m$ of  $\OF^\times$ to be the set of elements $x \in \OF^\times$ such that $v(x-1) \ge m$.

 We denote the unique unramified quadratic field extension of $F$ by $E$. Since $q$ is odd, we note that $E = F(\sqrt{\delta})$ for any element $\delta \in \OF^\times \setminus (\OF^\times)^2.$ We denote the ring of integers of $E$ by $\OF_E$. The valuation $v$ and the absolute value $| \ |$ naturally extend to the field $E$. Note that $\varpi$ is a uniformizer of $\OF_E$. We let $x \mapsto \bar{x}$ denote the unique non-trivial automorphism of $E/F$.

We let $\eta$ denote the unique unramified quadratic character of $F^\times$; equivalently, $\eta$ is the character associated to the extension $E/F$ via local class field theory. For each character $\chi$ of $F^\times$, we let $a(\chi)$ denote the smallest integer such that $\chi$ is trivial on the subgroup $U_{a(\sigma)}$. Similarly, for a character $\chi$ of $E^\times$, we let $a(\chi)$ denote the smallest integer such that $\chi$ is trivial on the subgroup $\{x \in \OF_E^\times: v(x-1) \ge a(\chi)\}$.

We fix once and for all an additive character $\psi$ of $F$ such that $\psi$ is trivial on $\OF$ but not on $\varpi^{-1}\OF$. We let $\psi_E:= \psi \circ \Tr_{E/F}$ be the corresponding additive character on $E$.

Throughout this section, we denote $G = \GL_2(F)$ and $K=\GL_2(\OF)$. Define subgroups
$N =
\{n(x):  x\in F \}$,
$A = \{a(y): y\in F^\times \}$,
$Z =\{ z(t):
t \in F^\times \}$, $B_1=NA$,
and $B = Z N A = G \cap
\left[
  \begin{smallmatrix}
    *&*\\
    &*
  \end{smallmatrix}
\right]$ of $G$.
 For each integer $r$,
denote
\[
K_1(r)
= K \cap \begin{bmatrix}
  1+ \p^r  & \mathfrak{o}  \\
  \p^r  & \mathfrak{o}
\end{bmatrix}, \quad K(r) = K \cap \begin{bmatrix}
  1+ \p^r  & \p^r  \\
  \p^r  & 1+\p^r
\end{bmatrix}, \quad B_1(r) = K \cap \begin{bmatrix}
  1+ \p^r  & \p^r  \\
  0 & 1
\end{bmatrix}.
\]

We note our normalization of Haar measures.
The measure $dx$ on the additive group $F$ assigns volume 1
to $\OF$, and transports to a measure on $N$.
The measure $d^\times y$ on the multiplicative group $F^\times$ assigns
volume 1 to $\OF^\times$,
and transports to measures on
$A$ and $Z$.
We obtain a left Haar measure $d_Lb$ on $B$ via
$d_L(z(u)n(x)a(y)) = |y|^{-1}\, d^\times u \, d x \, d^\times
y.$
Let $dk$ be the probability Haar measure on $K$.
The Iwasawa decomposition
$G = B K$ gives a left Haar measure $dg = d_L b \, d k$ on $G$.

For $\pi$ an irreducible admissible generic representation of $G$, we let $\omega_\pi$ denote the central character of $\pi$. We define $a(\pi)$ to be the smallest non-negative integer such that $\pi$ has a $K_1(\p^{a(\pi)})$-fixed  vector. It is known that $\pi$ can be realized as a unique subrepresentation of the space of functions $W$ on $G$ satisfying $W(n(x)g) = \psi(x)W(g)$. This is the Whittaker model of $\pi$ and will be denoted $\mathcal{W}(\pi,\psi)$.

If $\pi$ is unitary, there is a unique (up to multiples) $G$-invariant inner product $\langle, \rangle$ on it. In this case, for a vector $v_0 \in \pi$, we define the matrix coefficient $\Phi_{v_0}$ on $G$ as follows:
$$\Phi_{v_0}(g) =
\frac{\langle
\pi(g)v_0 ,  v_0\rangle}{\langle v_0 , v_0\rangle}$$ which is clearly unchanged if $v_0$ is multiplied by a constant and is also independent of the normalization of inner product. We will normalize the inner product in the model $\mathcal{W}(\pi,\psi)$ as follows: \begin{equation}\label{e:norminnerwhit}\langle W_1, W_2 \rangle = \int_{F^\times} W_1(a(t)) \overline{W_2(a(t))}d^\times t.\end{equation}

The following lemma will be useful for us.
\begin{lemma}\label{l:mini}Let $\pi$ be an irreducible admissible supercuspidal representation of $G$ such that $a(\omega_\pi) < a(\pi)/2$. Then $\pi$ is twist-minimal, i.e., $a(\pi \otimes \chi) \ge a(\pi)$ for each character $\chi$ of $F^\times$.
\end{lemma}
\begin{proof}Suppose, on the contrary, that $\pi \simeq \sigma \otimes \chi^{-1}$ with $\sigma$ minimal, and $a(\sigma) < a(\pi)$. As $\sigma$ and $\pi$ are supercuspidal, we have $2 \le a(\sigma) < a(\pi)$. By a result of Tunnell \cite[Prop. 3.4]{tunnell78}, we have $a(\pi) = a(\sigma \otimes \chi^{-1}) = 2 a(\chi)$; so $a(\chi)>1$. Since $q$ is odd, we have  that $a(\chi^2)=a(\chi) = a(\pi)/2$. As $a(\omega_\pi)<a(\pi)/2$, it follows that $a(\omega_\pi \chi^2) = a(\chi^2)$.  On the other hand, we have $\omega_\pi = \omega_\sigma \chi^{-2}$, i.e., $\omega_\sigma = \omega_\pi \chi^2$. Therefore, $a(\omega_\sigma) = a(\chi^2)=a(\chi) =a(\pi)/2 > a(\sigma)/2$, which contradicts Proposition 3.4 of \cite{tunnell78}.
\end{proof}

\subsection{Inert tori and their eigenvectors}
For  $\alpha$, $\beta$, $\gamma$ elements of $F$, denote $S = S_{\alpha, \beta, \gamma} = \mat{\alpha}{\beta/2}{\beta/2}{\gamma}$ and define  $$T_{\alpha, \beta, \gamma}:= \{g\in \GL_2(F):\:^tgSg=\det(g)S\}.$$

\begin{definition}\label{d:inertorus}A subgroup $T$ of $G$ is called an \emph{inert torus} if $T = T_{\alpha, \beta, \gamma}$ such that $\delta:= \beta^2 - 4 \alpha \gamma $ satisfies\footnote{Equivalently, $\delta$ is not a square in $F$ and $v(\delta)$ is even.} $E = F(\sqrt{\delta})$. An inert torus $T$ is said to be in \emph{canonical form} if $T = T_{\alpha, 0, 1}$ for some $\alpha \in \OF^\times$, $-\alpha \notin (\OF^\times)^2$.
 \end{definition}

 If $T=T_{\alpha, \beta, \gamma}$ is an inert torus, then the map \begin{equation}\label{explicitisotorus}x+ y \sqrt{\delta}/2 \mapsto \left(\mat{x+y\beta/2}{y\gamma}{-y\alpha}
{x-y\beta/2}\right)\end{equation} gives an explicit isomorphism from  $E^\times$ to $T$.  If $T$ is an inert torus in canonical form, then  that $\delta = -4 \alpha$ and \eqref{explicitisotorus} takes $\OF_E^\times$ isomorphically onto $T(\OF)= T \cap K$. It follows immediately that for an inert torus $T$ in canonical form we have $T = Z T(\OF) = \bigsqcup_{n \in \Z}\varpi^n T(\OF)$.

We note down several additional useful properties about inert tori.
\begin{proposition}\label{p:inertori}\begin{enumerate}\item All inert tori in $G$ are conjugate. \item Let $T$ be an inert torus. Then there exists $g \in G$ such that $gTg^{-1}$ is in canonical form.  \item If $T_1$, $T_2$ are two inert tori in canonical form, then there exists $y \in \OF^\times$ such that $T_1 = a(y)T_2a(y)^{-1}$.
    \item Let $T$ be an inert torus in canonical form. Then $G=B_1T = TB_1$ and $K = B_1(\OF)T(\OF)= T(\OF)B_1(\OF)$.
\end{enumerate}
 \end{proposition}
\begin{proof}All the parts of the above Proposition follow from elementary computations involving 2 by 2 matrices.
Let us start with part (3). If $T_1 = T_{\alpha_1, 0,1}$ and $T_2=T_{\alpha_2, 0, 1}$, then there exist $m \in \OF^\times$ such that $\alpha_2= m \alpha_1^2$. So $S_{\alpha_2, 0, 1} = a(m)S_{\alpha_1, 0, 1}a(m)$ and therefore $T_2 = a(m^{-1})T_1 a(m)$.

Next we prove part (2). Suppose that $T$ is associated to a matrix $S$. There exists  $h \in \GL_2(F)$ such that $ {}^th S h = \mat{\lambda_1}{0}{0}{\lambda_2}$ for some $\lambda_i \in F^\times$. Write $\lambda_1/\lambda_2 = m n^2$ with $m \in \OF^\times$. Then $\lambda_2^{-1}\mat{n^{-1}}{}{}{1} {}^th Sh \mat{n^{-1}}{}{}{1} = \mat{m}{0}{0}{1}$. Consequently, we have $(h(a(n^{-1}))^{-1}T (h(a(n^{-1})) = T_{m,0,1}$. Part (1) frollows from Parts (2) and (3).

Finally, let us prove part (4). For $g = \mat{a}{b}{c}{d}$, put $$u_1 = \frac{\alpha (ad-bc)}{\alpha a^2 + c^2},  \ m_1 = - \frac{ab\alpha + cd}{\alpha a^2 + c^2}, \ u_2 = \frac{c^2 + d^2 \alpha}{\alpha (ad - bc)}, \ m_2 = - \frac{ac + \alpha bd}{\alpha(ad - bc)}.$$ Then an easy calculation shows that $g \mat{u_1}{m_1}{}{1} \in T$ and $\mat{u_2}{m_2}{}{1}g \in T$. Furthermore, if $g \in K$ then it is immediate that $u_1, u_2 \in \OF^\times$, $m_1, m_2 \in \OF$.
\end{proof}

Now let $T\subset G$ be an inert torus and let $\theta:E^\times
\rightarrow \C^\times$ be a character such that
$\theta|_{F^\times} = 1$. Using the isomorphism
\eqref{explicitisotorus}, we view $\theta$ as a character of $T$
(note that this entails \emph{fixing} a choice of square root
$\sqrt{\delta}$ in $E^\times$ which we henceforth do without
comment). Let $\pi$ be an irreducible admissible generic
representation of $G$ with \emph{trivial central character}. A
non-zero vector $v \in \pi$ is said to be a $(T,
\theta)$-eigenvector
if $$\pi(t) v = \theta(t) v, \text{ for all } t \in T.$$

It is known  that the space of $(T, \theta)$-eigenvectors in $\pi$ has dimension less than or equal to 1, and it has dimension 1 if and only if the epsilon factor $\eps(1/2, \pi \otimes \mathcal{AI}(\theta^{-1}))$ (which is equal to $\pm 1$) equals 1, where $\mathcal{AI}(\theta^{-1})$ is the  representation of $G$ obtained from $\theta^{-1}$ by automorphic induction; see \cite{MR0401654}, \cite{Saito1993}, \cite{ralfnewform}.

The precise choice of $T$ is  unimportant, because any two inert tori are conjugate in $G$. If $T_1$, $T_2$ are two inert tori with $T_2 = g T_1 g^{-1}$, and $v_1$ is a $(T_1, \theta)$-eigenvector, then $\pi(g)v_1$ is a $(T_2, \theta)$-eigenvector.
In particular, we may assume,
by taking a suitable conjugate of $T$, that our inert torus
$T$ is in canonical form $T=T_{\alpha,0,1}$ (see part (2) of Proposition \ref{p:inertori}).
In this case, we have $\sqrt{\delta}=2\sqrt{-\alpha}$. We define the shorthand notation $$w_\alpha := \mat{0}{1}{-\alpha}{0}.$$ The isomorphism \eqref{explicitisotorus} now reads
\begin{equation}\label{eq:explicitisotorusCanonical}
x+y\sqrt{-\alpha}\mapsto x+yw_\alpha = \mat{x}{y}{-\alpha y}{x}.
\end{equation}

The goal of the rest of Section \ref{sec:localcalcs} is to delve
into a particularly important case
in which $(T, \theta)$-eigenvectors exist and explicate some remarkable properties in that case.

\subsection{Compact induction and minimal vectors}
\begin{definition}\label{d:KT}Given an inert torus $T=T_{\alpha, 0, 1}$ in canonical form, we define for each non-negative integer $r$, the congruence subgroup $K_T(r)$ of $K$ as follows:
$$K_T(r) = \{g = \mat{a}{b}{c}{d}\in K: a-d \in \p^r, \ c + b \alpha  \in \p^r \}=T(\OF)K(r).$$
\end{definition}
Using the calculations in the proof of Proposition \ref{p:inertori}, part (4) it can be seen that
\begin{equation} K_T(r) = T(\OF)B_1(r) = B_1(r)T(\OF). \end{equation} Since $B_1(r)$ intersects $T$ trivially, it follows that the index of $K_T(r)$ in $K$ is $\asymp q^{2r}$.
\begin{lemma}\label{lemmatexitence}Let $T=T_{\alpha,0,1}$ be an inert torus in canonical form. Let $\theta$ be a character of $E^\times$ such that $a(\theta)  = 2n$ and $\theta|_{F^\times} = 1$. 
Then there exists $a_{\theta, T} \in \OF^\times$ such that $\psi_E(\varpi^{-n} a_{\theta, T} \sqrt{-\alpha} u) = \theta(1 + \varpi^n u)$ for all $u \in \OF_E$.
\end{lemma}
\begin{proof} Note that $\psi'(x):= \theta(1 + \varpi^{n}\sqrt{-\alpha}x)$ is an additive character on $\OF$. So, there must exist $y \in F$ such that $\psi'(x) = \psi(xy)$ for all $x \in \OF$. Comparing conductors, we see that  $v(y) = -n$. So we may put $y = -2a_{\theta, T}\alpha \varpi^{-n}$ for some $a_{\theta, T} \in \OF^\times$.  We claim that this $a_{\theta, T}$ works. Indeed, let $u = a + b \sqrt{-\alpha}  \in \OF_E$, with $a,b \in \OF$. Then $\psi_E(\varpi^{-n} a_{\theta, T} \sqrt{-\alpha} u) = \psi(-2 \varpi^{-n} a_{\theta, T} \alpha b) = \psi'(b) = \theta(1 + \varpi^{n}\sqrt{-\alpha} b) = \theta(1 + \varpi^{n}u)$ where in the last step we have used that $\theta|_{F^\times} = 1$ and $a(\theta)=2n$.
\end{proof}

This enables the following definition.
\begin{definition}\label{d:chipi}
Given an inert torus $T=T_{\alpha, 0, 1}$ in canonical form and a character $\theta$ on $E^\times$ (which we view as a character of $T$) with $a(\theta)  = 2n$ and $\theta|_{F^\times} = 1$,
we extend the character $\theta$ to a function $\chi_{\theta, T}$ on the group $ZK_T(n) = TK(n) $ as follows: $$
\chi_{\theta, T}\left(t (1+\varpi^n g) \right)=\theta(t)\psi(\varpi^{-n}a_{\theta, T} \mathrm{Tr} (w_\alpha g)),$$
or equivalently
$$
\chi_{\theta, T}\left(t \mat{1+\varpi^n y_1}{\varpi^n x_1}{\varpi^n x_2}{1+\varpi^n y_2}\right)=\theta(t)\psi(\varpi^{-n}a_{\theta, T} (x_2 - \alpha x_1)).$$
\end{definition}
To see that the above formulae are well-defined, we note that $T\cap K(n)=1+\varpi^n \OF_E$ under the identification given by \eqref{eq:explicitisotorusCanonical}.

\begin{lemma}\label{keylemmamatrix}
The function $\chi_{\theta, T}$ is a multiplicative character of $ZK_T(n)$.
\end{lemma}
\begin{proof}
First, we claim that for all $k\in K(n)$, $t\in T$,
\begin{equation}\label{e:chikn}\chi_{\theta, T}(kt)=\chi_{\theta, T}(t)\chi_{\theta, T}(k)=\chi_{\theta, T}(tk).\end{equation}
To prove \eqref{e:chikn}, write $k = 1+ \varpi^n g$ and let $g' = t^{-1}g t$ so that $kt = t (1 + \varpi^n g')$.
Then
\begin{align*}
\chi_{\theta, T}(kt)&=\chi_{\theta, T}(t)\chi_{\theta, T}(1+\varpi^ng')\\&=\chi_{\theta, T}(t)\psi(\varpi^{-n}a_{\theta, T} \mathrm{Tr} (w_\alpha t^{-1} gt))\notag\\
&=\chi_{\theta, T}(t)\psi(\varpi^{-n}a_{\theta, T} \mathrm{Tr} (t^{-1} w_\alpha  gt))\notag\\
&=\chi_{\theta, T}(t)\psi(\varpi^{-n}a_{\theta, T} \mathrm{Tr} ( w_\alpha  g))\notag\\
&=\chi_{\theta, T}(t)\chi_{\theta, T}(k).\notag
\end{align*}
Next note that $\chi_{\theta, T}$ is multiplicative on the group $K(n)$ by using that $a(\theta)=2n$.
The multiplicativity of $\chi_{\theta, T}$ follows immediately by combining this fact with \eqref{e:chikn}.
\end{proof}

\begin{remark}Using $ZK_T(n) = TB_1(n)$, we can also write $\chi_{\theta, T}$ as
\begin{equation}
\chi_{\theta, T}\left(t \mat{ y}{\varpi^n x}{0}{1}\right) = \theta\left(t\left(1 + \sqrt{-\alpha}\varpi^n x/2\right)\right), \quad x \in \OF, y \in U_n.
\end{equation}
Further one can define the character $\chi_{\theta, T}$ on $Z K_T(n)$ directly  in terms of the entries of the matrix as follows:   $$\chi_{\theta, T}\left(\mat{a}{b}{c}{d}\right) = \theta\left((a \sqrt{-\alpha}+c)(ab\alpha + cd - 2\sqrt{-\alpha}(ad-bc)\right).$$

\end{remark}

By \cite{kutzko1} (see \cite{BH06} for a recent treatment), the supercuspidal representations of $G$ are obtained by compact induction from subgroups that are compact modulo $Z$. Precisely, let $\pi$ be an irreducible supercuspidal representation of $G$. Then there exists a maximal compact-mod-center subgroup $ZK$ of $G$, and an irreducible complex representation $\xi$ of $ZK$, such that $\pi \simeq c-\Ind_{ZK}^{G} \xi$ where $c-\Ind$ denotes compact induction \cite[15.5, 15.8]{BH06}. As shown in \cite{MR0507254}, the representation $\xi$ is itself induced from a smaller representation which is often one-dimensional. In the special case we are interested in, one can make all this very simple and explicit.
\begin{proposition}\label{p:existenceind}Let $\pi$ be a supercuspidal representation of $G$ with trivial central character satisfying $a(\pi) = 4n$ for some positive integer $n$.
There exists a character $\theta_\pi$ of $E^\times$  with $a(\theta_\pi)  = 2n$ and $\theta_\pi|_{F^\times} = 1$ such that for any inert torus $T$ in canonical form, we have
\begin{equation*}
\pi\simeq c-\Ind_{ZK_T(n)}^{G}\chi_{\theta_\pi, T}.
\end{equation*}
\end{proposition}
\begin{proof}This follows from the results of Kutzko \cite{kutzko1, MR0507254} but for our purposes it is more convenient to appeal to the treatment in \cite[Chapter 5]{BH06}. As $q$ is odd, $\pi$ is associated to a pair $(E/F, \chi)$ as in \cite[20.2]{BH06} and the assumption $a(\pi)=4n$ means that  the quadratic extension $E/F$  is unramified (and hence coincides with our setup) and furthermore that the integer $l(\chi)$ (in the terminology of \cite{BH06}) equals $2n-1$. Hence, defining $\theta_\pi=\chi$, the result follows from \cite[19.3]{BH06} (note that our character $\chi_{\theta_\pi, T}$ is denoted $\Lambda$ there).
\end{proof}

\begin{definition}\label{def:chipi}Given a supercuspidal representation $\pi$ of $G$ with  trivial central character satisfying $a(\pi) = 4n$ for some positive integer $n$, and an inert torus $T$ in canonical form, we let  $\chi_{\pi}$ denote the character $\chi_{\theta_\pi, T}$ of $ZK_T(n)$. Thus, \begin{equation*}
\pi\simeq c-\Ind_{ZK_T(n)}^{G}\chi_{\pi}.
\end{equation*}
\end{definition}

\begin{remark}\label{rmk:theta-vs-xi}
As $q$ is odd, and $a(\pi)$ is even, the representation $\pi$ is a dihedral supercuspidal representation associated to a character $\xi_\pi$ of $E^\times$ (see, e.g., \cite[Sec. 1.2]{ralfnewform}). Therefore, it is natural to ask for the relation between the characters $\theta_\pi$ and $\xi_\pi$. This is given by $\theta_\pi= \xi_\pi \eta_E$ where $\eta_E$ is the unique unramified  quadratic character on $E^\times$; see \cite[34.4]{BH06}.
\end{remark}

\begin{proposition}\label{p:existencewalds}Let $\pi$, $\theta_\pi$ be as in Proposition \ref{p:existenceind} and let $T$, $\chi_\pi$ be as in Definition \ref{def:chipi}.
Then there exists a unique up to multiples element $v\in \pi$ such that
\begin{equation}\label{e:optimaldef}
\pi(k)v=\chi_\pi(k)v, \text{\ for\ }k\in ZK_T(n).
\end{equation}
In particular $v$ is a $(T,\theta_\pi)$-eigenvector.
\end{proposition}
\begin{proof}
Recall that any element in $\pi = c-\Ind_{ZK_T(n)}^{G}\chi_\pi$ is a function $\phi$ on $G$ such that $\phi(kg)=\chi_\pi(k)\phi(g)$ for $k\in ZK_T(n)$, with the group $G$ acting by right translation. In particular we can take $\phi$ to be the function
\begin{equation}\label{eq:optiWaldsVCompactmodel}
\phi(g)=\begin{cases}
\chi_\pi(g),&\text{\ if \ }g\in ZK_T(n);\\
0, &\text{\ otherwise.}
\end{cases}
\end{equation}
Then it's clear that
\begin{equation}
\pi(k)\phi=\chi_\pi(k)\phi, \text{\ for\ }k\in ZK_T(n).
\end{equation}
The uniqueness assertion follows from the general fact that the
space of $(T, \theta)$-eigenvectors has dimension at most 1.
\end{proof}
The above Proposition allows us to make the following definition.
\begin{definition}\label{def:waldsopt}Let $\pi$ be a
  supercuspidal representation of trivial central character
  satisfying $a(\pi) = 4n$ for some positive integer $n$.  By a
  \emph{minimal vector} in $\pi$, we will mean a nonzero vector
  satisfying \eqref{e:optimaldef} for some inert torus $T$ in
  canonical form.
\end{definition}

As we have seen, minimal vectors exist. In fact, whenever we fix
an inert torus $T$ in canonical form, there is a unique up to
multiples $(T, \theta_\pi)$-eigenvector; we will call
such a vector a \emph{minimal vector for $T$}. By
part (2) of Proposition \ref{p:inertori}, it follows that the set of all
minimal vectors (without fixing $T$) lie in a single $A(\OF^\times)$-orbit.

As it turns out, minimal vectors have remarkable properties
which make them extremely special.  Indeed (as pointed out in
the introduction) a minimal vector may be viewed as the true
non-archimedean analogue of the lowest weight vectors in
(archimedean) holomorphic discrete series representations. As
shown in Section \ref{s:mc}, the matrix coefficient associated
to such a vector has the amazing property of being a
\emph{character} on the support. This implies that the
minimal vectors are those for which the associated matrix
coefficient function has smallest possible support. Another
important reason for singling out the $(T, \theta)$-eigenvectors
associated to the character $\theta=\theta_\pi$ is that the conductor
of the degree 4 $L$-function
$L(s, \pi \otimes \mathcal{AI}(\theta^{-1}))$ is smallest when
$\theta = \theta_\pi$.
\subsection{Main results}
For the rest of Section \ref{sec:localcalcs}, we let $\pi$ be a supercuspidal representation of trivial central character satisfying $a(\pi) = 4n$ for some positive integer $n$. Note that such a representation is automatically unitary. Our first result shows that minimal vectors have a remarkably simple description in the Whittaker model.

\begin{proposition}\label{Whittakermain}Let $W_0 \in
  \mathcal{W}(\pi,\psi)$be
   a minimal vector (with respect to some inert torus $T$ in canonical form) in the Whittaker model of $\pi$. Then 
the restriction of $W_0$ to $A$ is, for $\a=-a_{\theta, T}\alpha$ and some normalizing constant $c \in \C$, given by
$$W_0(a(y)) =\begin{cases} c & \text{ if } y \in \varpi^{-2n}\a U_n \\ 0& \text{ otherwise.} \end{cases}$$
\end{proposition}
\begin{proof}
We define an intertwining operator from $c-\Ind_{ZK_T(n)}^{G}\chi_{\pi}$ to $\mathcal{W}(\pi,\psi)$ via
\begin{equation}
\phi \mapsto W_\phi(g)=\int\limits_{F}\phi(\mat{-\frac{\varpi^{2n}}{a_{\theta, T}\alpha}}{0}{0}{1}\mat{1}{x}{0}{1}g)\psi(-x)dx.
\end{equation}
To see that the operator above is non-trivial, we compute directly the special values of the Whittaker function for the minimal vector, as defined in \eqref{eq:optiWaldsVCompactmodel}.
In particular
\begin{equation}
W_0(a(y))=\int\limits_{F}\phi(\mat{-\frac{\varpi^{2n}}{a_{\theta, T}\alpha}}{0}{0}{1}\mat{1}{x}{0}{1}\mat{y}{0}{0}{1})\psi(-x)dx.
\end{equation}
Recall that $\phi(g) = 0$ unless $g \in Z K_T(n)$. So to ensure that
$$ \mat{-\frac{\varpi^{2n}}{a_{\theta, T}\alpha}}{0}{0}{1}\mat{1}{x}{0}{1}\mat{y}{0}{0}{1}=\mat{-y\frac{\varpi^{2n}}{a_{\theta, T}\alpha}}{-x\frac{\varpi^{2n}}{a_{\theta, T}\alpha}}{0}{1}\in ZK_T(n), $$
we need $v(y)=-2n$, $-y\frac{\varpi^{2n}}{a_{\theta, T}\alpha}-1\in \p^n$ and $-x\frac{\varpi^{2n}}{a_{\theta, T}\alpha}\in \p^n$. The conditions on $y$ gives $y\in -\varpi^{-2n} a_{\theta, T}\alpha U_n$.  Thus $W_0(a(y))=0$ if $y\notin -\varpi^{-2n} a_{\theta, T}\alpha U_n$. On the other hand if $y\in-\varpi^{-2n} a_{\theta, T}\alpha U_n$,
$$ \mat{-y\frac{\varpi^{2n}}{a_{\theta, T}\alpha}}{-x\frac{\varpi^{2n}}{a_{\theta, T}\alpha}}{0}{1}=I_2-\varpi^n \mat{\varpi^{-n}+y\frac{\varpi^{n}}{a_{\theta, T}\alpha}}{x\frac{\varpi^{n}}{a_{\theta, T}\alpha}}{0}{0}.$$
By definition of $\phi$ in \eqref{eq:optiWaldsVCompactmodel} and Definition \ref{d:chipi},
\begin{align}
W_0(a(y))&=\int\limits_{v(x)\geq -n} \psi\circ\Tr(-\varpi^{-n}\mat{0}{a_{\theta, T}}{-a_{\theta, T}\alpha}{0}\mat{y\frac{\varpi^{n}}{a_{\theta, T}\alpha}+\varpi^{-n}}{x\frac{\varpi^{n}}{a_{\theta, T}\alpha}}{0}{0})\psi(-x)dx\\
&=\int\limits_{v(x)\geq -n} \psi\left(\Tr\mat{0}{0}{y+\varpi^{-2n}a_{\theta, T}\alpha}{x}\right)\psi(-x)dx \notag\\
&=\int\limits_{v(x)\geq -n} \psi(x)\psi(-x)dx\notag
\end{align}
is a non-zero constant independent of $y$ in the support.
\end{proof}

\begin{remark}Recall that different inert tori in canonical form are $A(\OF^\times)$ conjugate. Moreover, it is well known that a vector in  $\mathcal{W}(\pi,\psi)$ is uniquely specified by its restriction to $A$ (the so-called Kirillov model). Therefore, Proposition \ref{Whittakermain} gives us an alternative way to characterize minimal vectors: \emph{these are precisely those vectors which in the Kirillov model are equal to the characteristic function of $\varpi^{-2n}a U_n$ for some $a \in \OF^\times$.}\end{remark}

\begin{remark}Using \eqref{e:norminnerwhit}, it is clear that one can pick $c=\left(|\OF^\times/U_n|\right)^{1/2}$ in Proposition \ref{Whittakermain} for $W_0$ to be $L^2-$normalized.
\end{remark}
 Proposition \ref{Whittakermain} has some key consequences which will be crucial for our global results.
 \begin{corollary}\label{cor:whit}Let $T$ be an inert torus in canonical form and let $W_0 \in \mathcal{W}(\pi,\psi)$ be a minimal vector for $T$ in the Whittaker model. Let $\a \in \OF^\times$ be as in Proposition \ref{Whittakermain}. Let $g \in G$ and (using Proposition \ref{p:inertori}) write $g = \mat{y}{m}{}{1}t$ for $t \in T$, $y \in F^\times$, $m \in F$. Then  we have \begin{equation}\label{formulawhittaker}\frac{W_0(g)}{\langle W_0, W_0 \rangle^{1/2}} = |\OF^\times / U_n|^{1/2} \times \begin{cases} \theta_\pi(t) \psi(m)  & \text{ if } y \in \varpi^{-2n}\a U_n \\ 0& \text{ otherwise.} \end{cases}\end{equation}
\end{corollary}
\begin{proof}This is immediate as $W_0$ is a $(T, \theta_\pi)$-eigenvector.
\end{proof}

\begin{corollary}Let $W_0 \in \mathcal{W}(\pi,\psi)$ be a minimal vector in the Whittaker model of $\pi$. Then $$\frac{\sup_{g \in G} |W_0(g)|}{\langle W_0, W_0 \rangle^{1/2}} \asymp q^{n/2}.$$
\end{corollary}
\begin{proof}This is immediate from the previous Corollary.
\end{proof}

\begin{corollary}\label{localwsupportfinal}
Let $W_0 \in \mathcal{W}(\pi,\psi)$ be a minimal vector in the Whittaker model of $\pi$ and let $ k \in K$. Then there exists some $b \in \OF^\times/U_n$ such that
$$\frac{|W_0(a(y)k)|^2}{|\OF^\times/ U_{n}|}  = \begin{cases} \langle W_0, W_0 \rangle & \text{ if } y \in \varpi^{-2n}(b +
\p^n) \\ 0 &\text{ otherwise.}\end{cases}$$
\end{corollary}
\begin{proof}By assumption, $W_0$ is a $(T, \theta_\pi)$-eigenvector for some inert torus $T$ in canonical form. Using the last part of Proposition \ref{p:inertori}, we can write $k = \mat{z}{m}{}{1} t$ for $t \in T$, $z \in \OF^\times$, $m \in \OF$. So using Corollary \ref{cor:whit} we see that $\frac{|W_0(a(y)k)|^2}{|\OF^\times/ U_{n}|}$ equals $\langle W_0, W_0 \rangle$ if  $y \in \varpi^{-2n}z^{-1}\a U_n$ and equals 0 otherwise.
\end{proof}


\section{The QUE test vector property}\label{sec:test-vector-property}
Here we revisit
the discussion of
Section \ref{sec:intro-period-integrals-que}
in a local context,
and establish the local results
underlying the proof of Theorem \ref{thm:watson-ext}.

\subsection{Generalities}
We continue to use the notations of the previous section. In particular, the base field $F$ has odd residue characteristic (indeed, some of the results we will state below fail in the stated forms for even residual characteristic). Let $\pi_1, \pi_2, \pi_3$
be
generic irreducible unitary
representations
of $G$ with $\prod_{i=1}^3\omega_{\pi_i}=1$.
We assume that they arise
as local components of cuspidal automorphic representations;
this implies sufficient bounds towards temperedness
to give the absolute convergence
of the matrix coefficient integrals
\[
  \mathcal{H} : \pi_1 \otimes \pi_2
  \otimes \pi_3 \rightarrow \mathbb{C}
\]
\[
  \mathcal{H}(v_1,v_2,v_3)
  := \int_{g \in Z\bs G}
  \langle g v_1, v_1 \rangle
  \langle g v_2, v_2 \rangle
  \langle g v_3, v_3 \rangle
\]
for smooth vectors $v_i \in \pi_i$.
One calls $\pi_1 \otimes \pi_2 \otimes \pi_3$
\emph{distinguished}
if $\mathcal{H}$ is not identically zero.
By a result of Prasad \cite{Prasad:90a},
\begin{equation}\label{eq:prasad-criterion}
  \text{$\pi_1 \otimes \pi_2 \otimes \pi_3$ is distinguished}
  \iff
\eps(\pi_1 \otimes \pi_2 \otimes \pi_3,1/2) = 1.
\end{equation}

We focus here on the case in which
\begin{equation}
  \pi_1 = \overline{\pi_2} =: \pi
\end{equation}
and in which the conductor of $\pi$
is large compared to that of $\pi_3$.
This case is the relevant one when
considering the quantum unique ergodicity (QUE) problem for global automorphic
forms having $v \in \pi$ as a local component.
One then encounters,
after an application of Ichino's formula,
the local integrals
\begin{equation}\label{eq:H-local-integral-QUE}
  \mathcal{H}(v,\overline{v}, u),
\end{equation}
where
$u$ is an ``essentially fixed'' unit vector,
while
either the conductor of $\pi$ or the residue field cardinality
of $F$ tends off to $\infty$.  As explained at length in
\cite{NPS}, the size of
\begin{equation}\label{eq:normalized-local-QUE}
  C(\pi \otimes \overline{\pi })^{1/2}
  \mathcal{H}(v,\overline{v}, u)
\end{equation}
quantifies the relative difficulty of the QUE and subconvexity
problems.

When $a(\pi) = 1$ and $a(\pi_3) = 0$ and $v$ is a newvector,
it was shown in \cite{PDN-HQUE-LEVEL}
that the quantity
\eqref{eq:normalized-local-QUE}
has size $\asymp 1$.
This corresponds globally to the QUE and subconvexity problems
for a sequence of squarefree level newforms
having approximately equivalent difficulty.

It was observed in \cite{NPS} that
if
$a(\pi) \geq 2$, $a(\pi_3) = 0$ and $v \in \pi$
is an $L^2$-normalized
newvector,
then
\eqref{eq:normalized-local-QUE} is rather small; globally,
this says that the QUE
problem for newforms of powerful level is substantially \emph{easier} than the
corresponding subconvexity problem.
Related results were obtained in \cite{Hu:17a}
when $a(\pi_3) > 0$.

It is natural to ask whether the equivalence of difficulty in
the squarefree level case may be restored in the case of
powerful levels by choosing the test vector more carefully.
This was shown in \cite[Rmk 30, Rmk 50]{nelson-padic-que} when
$\pi$ belongs to the principal series
by taking for $v$ a ``$p$-adic microlocal lift.''
Below we address the case in which $\pi$ is supercuspidal,
assuming that its conductor
satisfies
the congruence condition from Section \ref{sec:localcalcs}. It turns out that a minimal vector works for this case.

\subsection{Matrix coefficients of minimal vectors}\label{s:mc}In this subsection, we assume that $\pi$ is a supercuspidal representation of $G$
with trivial central character and conductor of the form $a(\pi) = 4n$ for some positive
integer $n$. We look at the matrix coefficient associated to a minimal vector for $\pi$.

The matrix coefficients for representations before and after compact induction can be directly related; see, for example, \cite{KR14}. We briefly recall this relation. Let $H\subset G$ be an open and closed subgroup containing $Z$ with $H/Z$ compact. Let $\rho$ be an irreducible smooth representation of $H$ with unitary central character and $\pi=c-\Ind_H^G(\rho)$.  By the assumption on $H/Z$, $\rho$ is automatically unitarizable, and we shall denote the unitary pairing on $\rho$ by $\langle\cdot,\cdot\rangle_{\rho}$. Then one can define the unitary pairing on $\pi$ by
\begin{equation}\label{e:pairing}
  \langle\phi,\psi\rangle=\sum\limits_{x\in H\backslash G}\langle\phi(x),\psi(x)\rangle_{\rho}.
\end{equation}
If we let $y\in H\backslash G$ and $\{v_i\}$ be a basis for $\rho$, the elements
$$ f_{y,v_i}(g)=\begin{cases}
  \rho(h)v_i,&\text{\ if \ }g=hy\in Hy;\\
  0,&\text{\ otherwise.}
\end{cases}$$
form a basis for $\pi$.
\begin{lemma}\label{basislemmaapp}
  For $y,z\in H\backslash G$,
  \begin{equation}
    \langle \pi(g)f_{y,v_i},f_{z,v_j}\rangle=\begin{cases}
      \langle\rho(h)v_i,v_j\rangle_{\rho}, &\text{\ if\ }g=z^{-1}hy\in z^{-1}Hy;\\
      0,&\text{\ otherwise}.
    \end{cases}
  \end{equation}
\end{lemma}
\begin{proof}
  This is a direct consequence of \eqref{e:pairing} and the definition of our basis elements.
\end{proof}

\begin{proposition}\label{keymatrixprop}Let $v_0$ be a minimal vector in $\pi$ and let $\Phi_0(g) =
  \frac{\langle
    \pi(g)v_0 ,  v_0\rangle}{\langle v_0 , v_0\rangle}$.
  Then, \begin{equation}\Phi_0(g) = \begin{cases} \chi_\pi(g) &\text{ if } g \in Z K_T(n), \\ 0& \text{ otherwise.} \end{cases}\end{equation}
\end{proposition}
\begin{proof}This follows from putting $H=ZK_T(n)$, $\rho=\chi_\pi$, and $y=z=1$ in Lemma \ref{basislemmaapp} and using \eqref{eq:optiWaldsVCompactmodel}.
\end{proof}

\begin{remark} Thus, we see that the matrix coefficient of a minimal vector has the remarkable property that it is a character of its supporting group.
\end{remark}

\begin{corollary}\label{cor:matrixcoeff}Let $v_0$ be a minimal vector and let $\Phi_0(g) =
  \frac{\langle
    \pi(g)v_0 ,  v_0\rangle}{\langle v_0 , v_0\rangle}$. Let $\delta \asymp q^{-2n}$ be the volume of $K_T(n)$. Then $\int_{Z \bs G} |\Phi_0(g)|^2 dg = \delta$. Moreover, $R(\overline{\Phi_0})v_0 = \delta v_0$ and $\Phi_0 \ast \Phi_0 = \delta \Phi_0$ where we denote as usual  $$R(\overline{\Phi_0})v := \int_{Z\bs G} \overline{\Phi_0(g)}(\pi(g)v)  \ dg, \quad (\Phi_0 \ast \Phi_0)(h) := \int_{Z\bs G} \Phi_0(g^{-1})\Phi_0(gh) dg.$$
\end{corollary}
\begin{proof}This follows immediately from Proposition \ref{p:existencewalds} and Proposition \ref{keymatrixprop}.
\end{proof}

\subsection{The main result}

\begin{theorem}\label{t:localque}
  Assume that $\pi$ is an irreducible, admissible supercuspidal representation of $G$ with trivial central character and
  with conductor of the form $a(\pi) = 4n$ for some positive
  integer $n$.
  Let $v \in \pi$ be an $L^2$-normalized
  minimal vector. Let $\pi_3$ be an irreducible, admissible, unitary representation of $\GL_2(F)$ with trivial central character. \begin{enumerate}

  \item  We have $C(\pi \otimes \overline{\pi }) = q^{4n}$.

  \item Suppose that $u \in \pi_3$ is $K(n)$-fixed.
    Then
    \[
    \mathcal{H}(v,\overline{v}, u)
    =
    \vol(K_T(n))
    \int_{h \in T/Z}
    \langle h u, u \rangle
    = \vol(K_T(n))
    \int_{h \in T(\OF)}
    \langle h u, u \rangle
    \]
    where the $h$-integral
    is taken with respect
    to the probability Haar measure. In particular, if $u$ is also $T(\OF)$-fixed, then  \[
    C(\pi \otimes \overline{\pi })^{1/2}
    \mathcal{H}(v,\overline{v}, u)
    \asymp 1,
    \]
    with absolute implied constants.

  \item  Assume that
    \begin{equation}\label{eq:local-que-conductor-assumption}
      a(\pi) \geq 2 a(\pi_3).
    \end{equation}
    Then $\pi \otimes \overline{\pi } \otimes \pi_3$ is
    distinguished if and only if $a(\pi_3)$ is even.  Furthermore, whenever $a(\pi_3)$ is even, there exists a unit vector
    $u \in \pi_3$
    which is fixed by $K_T(n) = T(\OF)K(n)$, and hence (by the previous part) we have \[
    C(\pi \otimes \overline{\pi })^{1/2}
    \mathcal{H}(v,\overline{v}, u)
    \asymp 1.
    \]
  \end{enumerate}
\end{theorem}
\begin{proof}
  In our case, as $\pi$ has trivial central character, we have $ \overline{\pi }  \simeq \pi$. Therefore in the proof, we will replace $\overline{\pi}$ by $\pi$ whenever appropriate.

  First of all, $\pi$ is twist-minimal by Lemma \ref{l:mini}. The computations in \cite[Sec. 2.6]{NPS} now imply that $C(\pi \otimes \overline{\pi }) = q^{4n}$. This proves part (1).  Next, using Proposition \ref{keymatrixprop}, we see that $$\mathcal{H}(v,\overline{v}, u) = \int_{K_T(n)}\langle h u, u \rangle dh.$$
  Note that $K_T(n) = T(\OF)K(n)$ and by our normalization $T(\OF)$ has volume 1. So, if $u$ is $K(n)$-fixed, we obtain $$\mathcal{H}(v,\overline{v}, u) =  \vol(K_T(n))
  \int_{h \in T(\OF)}
  \langle h u, u \rangle
  $$ as required. This proves part (2) of the theorem.

  We now prove part (3). First of all, we verify that
  \begin{equation}\label{eqn:dist-vs-even-exp}
    \text{$\pi \otimes \pi  \otimes \pi_3$ is
      distinguished}
    \iff \text{$a(\pi_3)$ is even.}
  \end{equation}
  For this, we recall the three possibilities for $\pi_3$.
  \begin{enumerate}[(i)]
  \item $\pi_3$ is a principal series representation
    with trivial central character, hence induced by a pair
    $\{\chi,\chi^{-1}\}$ of characters of $F^\times$.
  \item $\pi_3$ is a twist of the Steinberg
    representation
    by a character $\chi$ of $F^\times$ satisfying $\chi^2=1$.
  \item $\pi_3$ is supercuspidal.
  \end{enumerate}

  In case (i),
  the conductor exponent $a(\pi_3) = a(\chi) + a(\chi^{-1})
  = 2 a(\chi)$
  is even. On the other hand, the self-duality of $\pi$ implies that $$\eps(\pi \otimes \pi \otimes \pi_3,1/2) = \eps(\pi \otimes \pi \otimes \chi,1/2)\eps(\pi \otimes \pi \otimes \chi^{-1},1/2)=1,$$ and therefore, using the criterion \eqref{eq:prasad-criterion}, we see that $\pi \otimes \pi  \otimes \pi_3$ is      distinguished.

  It remains to consider cases (ii) and (iii). We treat both cases simultaneously.
    Recall that the local Langlands
  correspondence associates to $\pi$ a Weil--Deligne
  representation of the form $\sigma_\xi := \Ind_E^F(\xi)$
  for the unramified quadratic extension $E/F$ and character $\xi$
  of $E^\times$  (cf. Remark \ref{rmk:theta-vs-xi}). The fact that $\pi$ has trivial central character implies that  the restriction of $\xi$ to $F^\times$ equals the unramified quadratic character on $F^\times$ (see, e.g., page 7 of \cite{ralfnewform}) and therefore \begin{equation}\label{e:galoischar}
  \xi^2(y) = \xi(x\overline{x})=1,
  \end{equation}
  for all $x \in E^\times$, $y \in F^\times$. Furthermore, $a(\pi) = 2 a(\xi)$ which leads to $a(\xi)=2n$.
  We denote also by
  $\sigma_3$ the Weil--Deligne representation associated to
  $\pi_3$.
  By rewriting Prasad's criterion \eqref{eq:prasad-criterion}
  in terms of Weil--Deligne representations,
  our task reduces to showing that
  \begin{equation}\label{eq:}
    \epsilon(\Ind_E^F(\xi) \otimes \Ind_E^F(\xi)
    \otimes \sigma_3,1/2) = 1
    \iff
    \text{$a(\pi_3)$ is even.}
  \end{equation}
  To compute these $\eps$-factors, we recall
  (see \cite[8.1.4]{Prasad:90a})
  that for any even dimensional Weil--Deligne
  representation $\sigma$,
  one has
  \begin{equation}\label{eq:sec3Basechangeepsilon}
    \epsilon(\Ind_E^F(\xi) \otimes \sigma,1/2 )=\epsilon(\sigma|_{E}\otimes\xi,1/2)\cdot \omega_{E/F}^{\frac{\dim \sigma}{2}}(-1).
  \end{equation}
  Moreover, denoting by $\xi^{-}$ the composition
  of $\xi$ with the nontrivial automorphism
  $x \mapsto \bar{x}$ of $E/F$,
  we have
  \begin{equation}\label{eq:sec3inductionrestriction}
    \Ind_E^F(\xi)|_{E}=\xi\oplus {\xi^-}.
  \end{equation}

    On the other hand, \eqref{e:galoischar} implies that $\xi^- = \xi^{-1} = \overline{\xi}$. Thus
    \begin{align}
      \epsilon(\Ind_E^F(\xi) \otimes \Ind_E^F(\xi) \otimes \sigma_3,1/2)&=\epsilon(\xi\otimes( \Ind_E^F(\xi) \otimes \sigma_3)|_{E},1/2)\\
                                                                                 &=\epsilon(\xi\otimes(\xi\oplus\overline{\xi}) \otimes \sigma_3|_{E},1/2) \notag \\
                                                                                 &=\epsilon(\xi^2\otimes\sigma_3|_{E},1/2)\epsilon(\sigma_3|_{E},1/2).
    \end{align}
    (The first equality follows from
    \eqref{eq:sec3Basechangeepsilon}
    applied to the four-dimensional
    Weil--Deligne representation $\Ind_E^F(\xi) \otimes
    \sigma_3$,
    the second from \eqref{eq:sec3inductionrestriction}.)

   By \eqref{eq:local-que-conductor-assumption}, $a(\xi) =  2n \geq a(\pi_3)$. On the other hand, as $n \ge 1$ and the residue characteristic of $F$ is odd, we have that $a(\xi) = a(\xi^2)$.  So $a(\xi^2 ) \geq a(\pi_3) > a(\pi_3)/2 +1  $ and hence  by \cite[Prop. 1.7 and Lemma 3.1]{Tunnell83}, the character $\xi^2$ appears in $\pi_3|_{E^\times}$ (where we think of $E^\times$ as a subgroup of $G$). So, by the main theorem of \cite{Tunnell83}, we have $\epsilon(\xi^2\otimes\sigma_3|_{E},1/2)=1$. (Observe here that $\sigma_3|_E$ corresponds, under local Langlands, to the base change of $\pi_3$ to $\GL_2(E)$).
      So, to finish the proof of \eqref{eqn:dist-vs-even-exp}, we need to show that  the quantity $\epsilon(\sigma_3|_{E},1/2)$ equals 1 if and only if $a(\pi_3)$ is even. For this, first observe that $\epsilon(\sigma_3|_{E},1/2)=\epsilon(\pi_3,1/2)\epsilon(\pi_3 \otimes \eta,1/2)$  where $\eta$ is the unique non-trivial unramified quadratic character. Now, by \cite[(11)]{ralfnewform}, we have $\epsilon(\pi_3 \otimes \eta,1/2) = (-1)^{a(\pi_3)}\epsilon(\pi_3,1/2)$ and hence $\epsilon(\sigma_3|_{E},1/2)=  (-1)^{a(\pi_3)}$, as desired.

 Finally, let $a(\pi_3) = 2 m$ for some nonnegative
    integer $m \leq n$.
    We now take for $u$ the Gross--Prasad test vector in \cite[Prop 2.6]{GrossPrasad:91a}
    defined relative to the torus $T$.
    Among other properties,
    this vector $u$ is invariant by  $Z K_T(m)$,
    hence in particular by $Z K_T(n)$, as required.
  \end{proof}


\section{Global cusp forms of minimal type}\label{s:global}
From now on, we move to a global setup. Throughout this section, the letter $G$ will stand for the algebraic group $\GL_2$.  We will usually denote a non-archimedean place $v$ by $p$ where $p$ is a rational prime. The set of all non-archimedean places (primes) will be denoted by $\f$. The archimedean place will be denoted by $v=\infty$.
Let $K_\infty = \SO_2(\R)$ be the standard maximal connected compact subgroup of $G(\R)$. We let $\psi$ denote the unique non-trivial additive character on $\A$ that is unramified at all finite places and equals  $e^{2 \pi i x}$ at $\R$. We normalize the Haar measure on $\R$ to be the Lebesgue measure. We fix measures on all our adelic groups  by taking the product of the local measures.  We give all discrete groups the
counting measure and thus obtain a measure on the appropriate quotient groups.
\subsection{Setup and statement of sup-norm result}\label{s:globalstatement} Let $\pi = \otimes_v \pi_v$ be an irreducible, unitary, cuspidal automorphic
representation of $G(\A)$ with trivial central character and the following additional property:
\begin{itemize}

\item If $\pi_p$ is ramified then $p$ is odd and $\pi_p$ is a supercuspidal representation satisfying $a(\pi_p) = 4n_p$ for some positive integer $n_p$.

\end{itemize}

We let $\c \subset \f$ denote the set of primes where $\pi_p$ is ramified. Let $N = \prod_{p\in \c} p^{n_p}$ and $C = N^4 = \prod_{p\in \c} p^{4n_p}$. Thus $C$ is the conductor of the representation $\pi$.

Since $\pi$ has trivial central character, there are two possibilities for $\pi_\infty$.

\medskip

\textbf{Case 1: Principal series representations.}
In this case, $\pi_\infty \simeq \chi_1 \boxplus \chi_2,$ where $\chi_1(y) = |y|^{it} \sgn(y)^{m}$, $\chi_2(y) = |y|^{-it} \sgn(y)^{m}$, with $m \in \{0,1\}$,  $t \in \R \cup (-\frac{i}{2},\frac{i}{2})$. In this case, put $$k=0, \quad T=1+|t|.$$

\textbf{Case 2: Holomorphic discrete series representations.} In this case $\pi_\infty$ is the unique irreducible subrepresentation of $\chi_1 \boxplus \chi_2$, where $\chi_1(y) = |y|^{\frac{k-1}{2} }$, $\chi_2(y) = |y|^{-\frac{k-1}{2}}$ for some positive even integer $k$. In this case we put $$T= k.$$

\medskip

In either case, we will call $k$ the lowest weight. Note that $k=0$ in Case 1. We say that a vector $\phi_\infty$ in $\pi_\infty$ is a lowest weight vector if \begin{equation}\label{weighteq}\pi_\infty \mat{\cos(\theta)}{\sin(\theta)}{-\sin(\theta)}{\cos(\theta)} \phi_\infty  = e^{ik \theta}\phi_\infty.\end{equation}

\begin{definition}
  We say in what follows that a non-zero automorphic form
  $\phi \in V_\pi$ is of ``minimal type'' if $\phi$ is a
  factorizable vector $\phi= \otimes_v \phi_v$ with
  $\phi_v \in V_{\pi_v}$
  that is lowest weight at
  the archimedean place and minimal at the finite places. Precisely:
  \begin{enumerate}
  \item For all $p \in \c$, $\phi_p$ is a minimal vector in the sense of Definition \ref{def:waldsopt}.

  \item For all $p \in \f$, $p \notin \c$, $\phi_p$ is $G(\Z_p)$-invariant.

  \item $\phi_\infty$ is a lowest weight vector.

  \end{enumerate}
\end{definition}
We define
$\|\phi\|_2 = \int_{Z(\A)G(F)\bs G(\A)} |\phi(g)|^2 dg.$

\begin{remark}\label{r:adelic}It is interesting to translate things to a classical setup. Suppose that $\phi$ is an automorpic form of  minimal type. By definition, for each $p \in \c$, $\phi_p$ is an minimal vector with respect to some inert torus in canonical form $T_p = T_{\alpha_p, 0,1}$ (as in Definition \ref{d:inertorus}) where $\alpha_p \in \Z_p^\times$; let $\chi_{\pi_p}$ be the character on $Z_p K_{T_p}(n_p)$ as defined in Definition \ref{d:chipi}. Let $D$ be an integer such that $D \equiv \alpha_p \pmod{p^{n_p}}$ for all $p \in \c$ and define the congruence subgroup $\Gamma_{T,D}(N)$ of $\SL_2(\Z)$ as follows:
$$\Gamma_{T,D}(N) = \left\{ \mat{a}{b}{c}{d} \in \SL_2(\Z): a \equiv d \pmod{N}, \ c \equiv -bD \pmod{N} \right \}.$$
Clearly, the group $\Gamma_{T,D}(N)$ contains the principal congruence subgroup $\Gamma(N)$. Define a character $\chi$ on $\Gamma_{T,D}(N)$ by $\chi(\gamma) = \prod_{p|N}\chi_{\pi_p}^{-1}(\gamma)$. Note that $\chi$ is trivial on the principal congruence subgroup $\Gamma(N^2)$ but non-trivial on $\Gamma(Nm)$ for any $1\le m <N$, $m|N$.

Then, the function $f$ on $\H$ defined by $f(x+iy) = y^{-k/2}\phi\left(\mat {y^{1/2}} {xy^{-1/2}}{}{y^{-1/2}} \right)$ has the following properties:

\begin{itemize}

  \item If we are in Case 1, then $f$ is a real analytic function satisfying  $\Delta f = - \lambda f $ and if we are in Case 2 then $f$ is a holomorphic function.

\item For all $\gamma \in \Gamma_{T,D}(N)$, $z \in \H$, \begin{equation}\label{masform} f|_k \gamma = \chi(\gamma)f.\end{equation}

\item $f$ decays
rapidly at the cusps.

 \item $f$ is an eigenfunction of all  the Hecke operators $T_n$ for $(n, N)=1$.

      \end{itemize}

It is also clear that $\sup_{g \in G(\A)} |\phi(g) | = \sup_{z\in \H} |y^{k/2}f(z)|$.
\end{remark}

Let the real numbers $\lambda_\pi(n)$ be the coefficients of the (finite part of the) $L$-function attached to $\pi$, i.e., \begin{equation}\label{e:lambdapi}L_\f(s, \pi) = \sum_{n=1}^\infty \frac{\lambda_\pi(n)}{n^s}.\end{equation} Note that all our $L$-functions are normalized so that the functional equation takes $s \rightarrow 1-s$.

\begin{definition}\label{d:deltapi}We fix $\delta_\pi$ to be any real number such that $\lambda_\pi(n) \le d(n)n^{\delta_\pi}$ for all positive integers $n$ where $d(n)$ is the divisor function. In particular, we may uniformly take $\delta_\pi = \frac{7}{64}$ in Case 1, and $\delta_\pi=0$ in Case 2.
\end{definition}

Our main result is as follows.
\begin{theorem}\label{t:globalmain}Let $\phi \in V_\pi$ be of minimal type and satisfy $\|\phi\|_2 = 1$. \begin{enumerate} \item If we are in Case 1 then \begin{equation}\label{E:CASE1RESULT}C^{\frac{1}{8} -\eps} T^{\frac{1}{6} - \eps} \ll_{\eps} \sup_{g \in G(\A)} |\phi(g)| \ll_{\eps} C^{\frac{1}{8} + \eps} T^{\frac12 + \eps} \min(C^{\frac{\delta_\pi}{2}} T^{\delta_\pi}, C^{\frac{1}{32}}).\end{equation}

\item If we are in Case 2, then \begin{equation}\label{E:CASE2RESULT}C^{\frac{1}{8} -\eps} k^{\frac{1}{4} - \eps} \ll_{\eps} \sup_{g \in G(\A)} |\phi(g)| \ll_{\eps}  C^{\frac{1}{8} + \eps} k^{\frac14+\eps}.\end{equation}
\end{enumerate}
\end{theorem}

We will prove this theorem by carefully looking at the Whittaker expansion. Before getting into the details of the proof, let us make a simple but key reduction. Let $\F$ be the subset of $B_1(\R)^+$ defined by $\F:= \{n(x)a(y): x\in \R, \ y\ge \sqrt{3}/2 \}.$ Let $$\J_N = \prod_{p|N} G(\Z_p).$$ Then, using strong approximation, it follows that for any $g \in G(\A)$, the double coset $G(\Q)g\prod_{p \nmid N}G(\Z_p)$ has a representative in $\J_N \times \F$. Since $\phi$ is left $G(\Q)$-invariant and right $\prod_{p \nmid N}G(\Z_p)$-invariant, it suffices in Theorem \ref{t:globalmain} to only consider the supremum for $g$ lying in  $\J_N \times \F$, i.e., $g= g_\f n(x)a(y)$ with $g_\f \in \J_N $, $n(x)a(y) \in \F$.

\subsection{Generalities on the Whittaker expansion and proof of the lower bounds}\label{s:fourierglobal}
Let $\pi$, $\phi$ be as in the statement of Theorem \ref{t:globalmain}. Let $g_\f = \prod_{v\in \f}g_v \in G(\A_\f)$, $x \in \R$, $y \in \R^+$. Then the Whittaker expansion for $\phi$ says that \begin{equation}\label{whittakerexp}
\phi(g_\f n(x)a(y) )=
\sum_{q \in \mathbb{Q}_{\neq 0}}
W_\phi(a(q) g_\f n(x)a(y))\end{equation} where $W_\phi$ is the global Whittaker newform
corresponding to $\phi$ given explicitly by
\begin{equation}\label{e:globwhit}W_\phi(g) = \int_{x \in \mathbb{A} / \mathbb{Q}} \phi(n(x) g)
\psi(-x) \, d x.\end{equation} For each unramified prime $p$, i.e., for $ p \in \f - \c$, let the function $W_p(g)$ on $G(\Q_p)$ be equal to the unique right $G(\Z_p)$-invariant function in the Whittaker model of $\pi_p$ normalized so that $W_{\pi_p}(1)=1.$ It is well-known that for $(m,N)=1$ we have $$m^{1/2} \prod_{p \in \f - \c} W_p(a(m)) = \lambda_\pi(m),$$ where $\lambda_\pi(m)$ is defined by \eqref{e:lambdapi}.
  For each ramified prime $p$, i.e., for $p \in \c$, let the function $W_p(g)$ on $G(\Q_p)$ be equal to $\frac{W_{0,p}(g)}{\langle W_{0,p}, W_{0,p} \rangle^{1/2}}$ where $W_{0,p}$ is an element corresponding to $\phi_p$ in the Whittaker model for $\pi_p$. The function $W_p(g)$ in this case is given explicitly by the right hand side of \eqref{formulawhittaker}.
Finally for $v=\infty$,  let the function $W_\infty(g)$ on $G(\R)$ be the element of the Whittaker model of $\pi_\infty$  corresponding to $\phi_\infty$, normalized so that $W_{\infty}(a(y))=\kappa(y)$ for all $y \in \R$ where  \begin{equation}\kappa(y):= \begin{cases} |y|^{1/2} K_{it}(2 \pi |y|) \sgn(y)^m &\text{ in Case 1,}\\  y^{k/2} e^{-2 \pi y} \left(\frac{1 + \sgn(y)}{2}\right) &\text{ in Case 2.}\end{cases}\end{equation} Put $$c_\infty = \langle W_\infty, W_\infty \rangle^{1/2}= \left( \int_{\R^\times} |\kappa(y)|^2  \frac{dy}{|y|}\right)^{1/2}.$$ It is a well-known fact (see, e.g.,  \cite[Lemma 5.3]{templier-large} or \cite[(27)]{sahasupwhittaker}) that \begin{equation}\label{hpiformula}\frac{\sup_{g \in G(\R)} |W_\infty(g)|}{c_\infty} = \frac{\sup_{y >0} \kappa(y)}{c_\infty} \asymp \begin{cases}T^{1/6} &\text{ in Case 1,} \\ k^{1/4}  &\text{ in Case 2.} \end{cases}\end{equation}

By Lemma 2.2.3 of \cite{michel-2009}, the function $W_\phi$ factors as follows. For $g_\f = \prod_{v\in \f}g_v \in G(\A_\f)$, $x \in \R$, $y \in \R^+$, we have \begin{equation}\label{whitfactor}W_\phi(g_\f n(x)a(y)) = \sqrt\frac{2 \zeta(2)}{L_{\f}(1, \pi, \Ad)} \times \frac{e^{2 \pi i x} \kappa(y)}{c_\infty} \times  \prod_{p \in \f}W_p(g_p) \end{equation} where $L_{\f}(1, \pi, \Ad) = \prod_{p<\infty} L(1,\pi_p ,\Ad)$ denotes the finite part of the global adjoint $L$-function for $\pi$. By a result of Hoffstein-Lockhart \cite{HL94}, we have \begin{equation}\label{e:hlbd}(CT)^{-\eps} \ll_{ \eps} L_{\f}(1, \pi, \Ad) \ll_{ \eps}(CT)^\eps.\end{equation}

\begin{remark} To deduce \eqref{whitfactor} from Lemma 2.2.3 of \cite{michel-2009}, note that  from Table 1 of \cite{NPS} that $\frac{L(1,\pi_p ,\Ad)\zeta_p(1)} {\zeta_p(2)}
= 1$ for all $p \in \c$.
\end{remark}

Using \eqref{e:globwhit},  \eqref{hpiformula}, \eqref{whitfactor}, \eqref{e:hlbd}, we conclude that \begin{align*}\sup_{g\in G(\A)}|\phi(g)| & \gg \sup_{g\in G(\A)}|W_{\phi}(g)| \\ &\gg_{ \eps} (CT)^{-\eps} h(\pi_\infty) \prod_{p \in \c}\sup_{g\in G(\Q_p)} |W_p(g)|   \end{align*}
where $h(\pi_\infty) = T^{1/6}$ in Case 1 and $h(\pi_\infty) = k^{1/4}$ in Case 2. By Corollary~\ref{cor:whit}, we have $$\prod_{p \in \c}\sup_{g\in G(\Q_p)} |W_p(g)| \gg_\eps C^{1/8 - \eps}.$$ This completes the proof of the lower bounds in Theorem \ref{t:globalmain}!

Next, recall that for $(m,N)=1$, we have
$\lambda_\pi(m) = m^{1/2} \prod_{p \in \f - \c} W_p(a(m)).$ From Definition \ref{d:deltapi}, we have \begin{equation}\label{thetaineq}\lambda_\pi(m) \ll_\eps m^{\delta_\pi+\eps}.\end{equation}

We will need the following property of the coefficients $\lambda_\pi(n)$ to get an improved bound in Case 1.
\begin{proposition}\label{p:sym4}Let $1 \le r \le 4$ be an integer. Then
$$\sum_{1 \le |n| \le X} |\lambda_\pi(n)|^{2r} \ll_\eps X (NTX)^\eps.$$
\end{proposition}
\begin{proof}This follows by first taking the $\sym^r$-lift of $\pi$ to $\GL_{r+1}$ which is known to exist \cite{MR533066, MR1937203} and then using the analytic properties of $L(s, \sym^r \pi \otimes \sym^r \bar{\pi})$.  For a detailed proof in the case $r=2$, we refer the reader to \cite[Lemma 2.1]{HL94}. The proofs in the other cases are essentially identical.
\end{proof}

 Let $g_\f \in \J_N$.  For each $m\in \Z$, we define $$\lambda'(m; g_\f) = \prod_{p\in \c}W_p\left(a(m/N^2)g_p\right).$$ By Corollary \ref{localwsupportfinal}, there exists some integer $b=b(g_\f)$ coprime to $N$, such that
 \begin{equation}\label{lambdaram} |\lambda'(m; g_\f)| = \begin{cases} \sqrt{\varphi(N)} &\text{ if }m \equiv b \pmod {N}\\ 0 &\text{ otherwise.}\end{cases}\end{equation}

 Therefore, for any $g_\f \in \J_N$ and $x \in \R$, $y \in \R^+$, the expansion \eqref{whittakerexp} together with the above discussion gives us:

 \begin{equation}\label{whittakerexp2}
\phi(g_\f n(x)a(y) )=
\sqrt\frac{2 \zeta(2)}{L_{\f}(1, \pi, \Ad)} \times \frac{1}{c_\infty} \sum_{\substack{m \in \Z \\ m \equiv b\bmod{N}}} m^{-1/2} e^{\frac{2 \pi i m x}{N^2}} \kappa(my/N^2) \lambda_\pi(m) \lambda'(m;g_\f).\end{equation}

In particular, the Whittaker expansion of $\phi$ is supported on an arithmetic progression! It is this remarkable feature that will allow us to prove a strong upper bound.  As a key first step, using \eqref{e:hlbd} and the triangle inequality, we note the bound

 \begin{equation}\label{whittakerexp3}
|\phi(g_\f n(x)a(y) )| \ll_\eps (CT)^\eps
 \frac{N^{1/2}}{c_\infty} \sum_{\substack{m \in \Z \\ m \equiv b\bmod{N}}} m^{-1/2} |\kappa(my/N^2)| |\lambda_\pi(m)|.\end{equation}
\subsection{Proof of the upper bounds}
 We can now prove the upper bounds in Theorem \ref{t:globalmain}. Throughout this subsection, let $g_\f \in \J_N$ and $x \in \R$, $y \in \R^+$, with $y \ge \frac{\sqrt{3}}2$. As noted at the end of Section \ref{s:globalstatement}, it is sufficient to restrict to $g=g_\f n(x) a(y)$ with $g_\f$, $x,y$ as above.

 First, we deal with Case 1. In this case we have $$|\kappa(my/N^2)|=N^{-1}(my)^{1/2} |K_{it}(2 \pi |my|/N^2)|.$$ By \cite[Lemma 5.3]{templier-large}, $c_\infty \gg e^{-\pi t/2}$. So \eqref{whittakerexp3} gives
  \begin{align}\label{whittakerexp5}
|\phi(g_\f n(x)a(y) )| \ll_\eps (CT)^\eps e^{\pi t/2}
  \left(\frac{y}{N} \right)^{1/2}
  \sum_{\substack{m \in \Z \\ m \equiv b\bmod{N}}} |\lambda_\pi(m)||K_{it}(2 \pi |my|/N^2)|\end{align} We need to prove the following two bounds:
 \begin{equation} \label{bound1} |\phi(g_\f n(x)a(y) )| \ll_\eps  C^{\frac{1}{8} + \frac{\delta_\pi}{2} + \eps} T^{\frac12 + \delta_\pi+\eps} \end{equation}
 \begin{equation} \label{bound2} |\phi(g_\f n(x)a(y) )| \ll_\eps  C^{\frac{1}{8} + \frac{1}{32} + \eps} T^{\frac12 + \eps} \end{equation}

Let $f(y) = \min(T^{1/6}, \left| \frac{y}{T} - 1\right|^{-1/4}).$ Then it is known  that $ e^{\pi t/2} |K_{it}(y)| \ll T^{-1/2} f(y)$; see, e.g., \cite[(3.1)]{templier-sup-2}. Furthermore, the quantity
 $\lambda_\pi(m)|K_{it}(2 \pi |my|/N^2)|$ decays exponentially for $ m \gg R$ where $R=\frac{N^{2+\eps}(T+T^{1/3 +\eps})}{2 \pi y}$. Therefore, if $R\ll 1$, the right side of \eqref{whittakerexp5} is negligible. So we henceforth assume that $R\gg 1$, i.e., $y \ll N^{2+\eps} T^{1+\eps}$. Furthermore, for the same reason, we can restrict the sum in \eqref{whittakerexp5} to $|m| <R$.

  Let us now prove \eqref{bound1}. We obtain from \eqref{thetaineq} and \eqref{whittakerexp5}

 \begin{align*}\label{whittakerexp6}
|\phi(g_\f n(x)a(y) )| &\ll_\eps (CT)^\eps T^{-1/2}
  \left(\frac{y}{N} \right)^{1/2}
  \sum_{\substack{1 \le |m| \le R \\ m \equiv b\bmod{N}}} m^{\delta_\pi + \eps} f(2 \pi |my|/N^2)\\&\ll_\eps (CT)^\eps T^{-1/2}
  \left(\frac{y}{N} \right)^{1/2}
  N^{\delta_\pi}\sum_{\substack{0 < |m| \le R/N \\ m \in \frac{b}N+\Z}} m^{\delta_\pi + \eps} f(2 \pi |my|/N)\\& \ll_\eps (CT)^\eps T^{-\frac12}
  y^{\frac12}
  N^{\delta_\pi-\frac12}\max(1, (R/N)^{\delta_\pi})\left(T^{\frac16} + \int_0^{\frac{R}{N}} \left|\frac{2 \pi xy}{NT} - 1 \right|^{-\frac14}dx\right)\\&\ll_\eps (CT)^\eps T^{-\frac12}
  y^{\frac12}
  N^{\delta_\pi-\frac12}\max(1, (R/N)^{\delta_\pi})\left(T^{\frac16} + \frac{NT}{y}\right)\\&\ll_\eps (CT)^\eps N^{\frac12 + 2 \delta_\pi}T^{\frac12+\delta_\pi},
  \end{align*}

 where in the last step we have used $1 \ll y  \ll N^2 T^{1+\eps}$. This completes the proof of \eqref{bound1} since $N^{\frac12 + 2 \delta_\pi}= C^{\frac18 +  \frac{\delta_\pi}2}.$

Let us now prove \eqref{bound2}. We obtain from Proposition \ref{p:sym4} and \eqref{whittakerexp5}, together with Holder's inequality:
 \begin{align*}
|\phi(g_\f n(x)a(y) )| &\ll_\eps (CT)^\eps T^{-1/2}
  \left(\frac{y}{N} \right)^{1/2}  \\ & \quad \times
  \left(\sum_{\substack{1 \le |m| \le R }} |\lambda_\pi(m)|^8 \right)^{1/8} \times \left(\sum_{\substack{0 < |m| \le R \\ m \equiv b\bmod{N}}}f(2 \pi |my|/N^2)^{8/7}\right)^{7/8}\\&\ll_\eps (CT)^\eps T^{-1/2}
  \left(\frac{y}{N} \right)^{1/2} R^{1/8}
  \left(\sum_{\substack{1 \le |m| \le R/N \\ m \in \frac{b}N+\Z}}  f(2 \pi |my|/N)^{8/7}\right)^{7/8}\\& \ll_\eps (CT)^\eps T^{-\frac12}
  y^{\frac12}
  N^{-\frac12}R^{1/8}\left(T^{\frac{4}{21}} + \int_0^{\frac{R}{N}} \left|\frac{2 \pi xy}{NT} - 1 \right|^{-\frac27}dx\right)^{7/8}\\&\ll_\eps (CT)^\eps T^{-\frac38}
  y^{\frac38}
  N^{-\frac14}\left(T^{\frac16} + \left(\frac{NT}{y}\right)^{7/8}\right)\\&\ll_\eps (CT)^\eps N^{\frac58}T^{\frac12},
  \end{align*}
which is equivalent to \eqref{bound2}.

  Next, we deal with Case 2. In this case, we have $\delta_\pi=0$ and $$|\kappa(my/N^2)|=N^{-k}(my)^{k/2} e^{-2 \pi my/N^2}.$$ By \cite[Lemma 5.3]{templier-large}, $c_\infty$ equals $(4 \pi)^{-k/2}\Gamma(k)^{1/2}$. So \eqref{whittakerexp3} gives
\begin{align*}
|\phi(g_\f n(x)a(y) )| \ll_\eps (CT)^\eps
 \frac{(4\pi y)^{k/2}N^{-k+1/2}}{\Gamma(k)^{1/2}} \sum_{\substack{m \in \Z \\ m \equiv b\bmod{N}}} e^{-2 \pi my/N^2} m^{(k-1)/2 +\eps} \\  \ll_\eps (CT)^\eps \left(\frac{y}{N} \right)^{1/2}
 \frac{2^{k/2}}{\Gamma(k)^{1/2}} \sum_{n\in \frac{b}{N} + \N} e^{-2 \pi ny/N} (2\pi ny/N)^{(k-1)/2 +\eps}.\end{align*}

To estimate the above sum we proceed similarly to \cite{xia}. Indeed, if we take the relevant series in \cite{xia} and replace $y \mapsto y/N$, $k \mapsto k/2$, and take the summation over $b/N+ \Z_{\ge0}$ instead of $\Z_{>0}$, we get our series above. Observe also that the function $\xi \mapsto e^{-2 \pi \xi y/N} (2\pi \xi y/N)^{(k-1)/2 +\eps}$ obtains its maximum at $$\xi = \frac{k/2 - 1/2 + \eps}{2 \pi (y/N)}.$$ So the general term in the series above is decreasing if $\xi < b/N$, i.e., if $\frac{y}{N} \gg \frac{Nk}{b}$. Now the argument of \cite{xia}, \emph{mutatis mutandis}, leads to

\begin{equation}
|\phi(g_\f n(x)a(y) )| \ll_\eps \begin{cases}(CT)^\eps \left(\frac{k^{1/4+\eps}}{(y/N)^{1/2}} + \frac{k^\eps  (y/N)^{1/2}}{k^{1/4}} \right) & \text{ if } \frac{y}{N} \ll \frac{Nk}{b} \\
 (CT)^\eps \left(\frac{k^{1/4+\eps}}{(y/N)^{1/2}} + \frac{k^{1/4+\eps}  N^{1/2}}{b^{1/2}} \right) &\text{ if } \frac{y}{N} \gg \frac{Nk}{b}.\end{cases}\end{equation} As $y \gg 1$, in either case we have $$|\phi(g_\f n(x)a(y) )| \ll_\eps C^{1/8 + \eps}k^{1/4 + \eps},$$ completing the proof in Case 2.

\subsection{The proof of Theorem \ref{thm:watson-ext}}

We explain in this final subsection the proof of Theorem
\ref{thm:watson-ext}. Let the notations be as in the statement
of Theorem  \ref{thm:watson-ext} and let $\sigma_g = \otimes_p
\sigma_p$.
Ichino's generalization of Watson's
formula~\cite{MR2449948} reads
  $$\frac{
      \left\lvert \int_{\Gamma \backslash \mathbb{H}}
        g(z) |f|^2(z) y^k \, \frac{d x \, d y}{y^2}
      \right\rvert^2
    }{
      \left(\int_{\SL_2(\Z) \backslash \mathbb{H}}
        |g|^2(z)  \, \frac{d x \, d y}{y^2} \right)
      \left( \int_{\Gamma \backslash \mathbb{H}}
        |f|^2(z) y^k \, \frac{d x \, d y}{y^2}
      \right)^2
    }
    =  \frac{1}{8 }
    \frac{\Lambda(\pi \times \pi \times \sigma_g,\onehalf)}{
      \Lambda(\ad \sigma_g,1) \Lambda(\ad \pi,1)^2} \ I_\infty \prod_{p| C}  I_p.
$$
 By an explicit calculation, the archimedean quantity $I_\infty$ is equal to 1 in our case; see \cite{watson-2008}. The local quantities $I_p$ are defined for each prime $p|C$ as follows:

  $$I_p =   \left(
    \frac{ L (\pi_p  \times \pi_p  \times \sigma_p, \onehalf ) \zeta_p
      (2)^2}{
      L (\ad \sigma_p, 1 )
      L (\ad \pi_p , 1 ) ^2 }
  \right) ^{-1}
   \mathcal{H}_p(v_p,\overline{v}_p, u_p)$$
  where we are using the notation of Section \ref{sec:test-vector-property}, and $v_p$ denotes the minimal vector in $\pi_p$, and $u_p$ denotes the unramified vector in $\sigma_p$. In particular, $u_p$ satisfies the condition in part (2) of Theorem \ref{t:localque} and therefore we have $$ \mathcal{H}_p(v_p,\overline{v}_p, u_p) \ \rm{Cond}(\pi_\p
\times \pi_p)^{1/2} \asymp 1.$$ On the other hand, it follows from well-known bounds on the Satake parameters that $$  \frac{ L (\pi_p  \times \pi_p  \times \sigma_p, \onehalf ) \zeta_p
      (2)^2}{
      L (\ad \sigma_p, 1 )
      L (\ad \pi_p , 1 ) ^2 } \asymp 1.$$ Therefore $I_p \rm{Cond}(\pi_\p
\times \pi_p)^{1/2} \asymp 1$ as required.

\bibliography{HNS-minimal}

\def\cprime{$'$} \def\cprime{$'$} \def\cprime{$'$}
\begin{thebibliography}{10}

\bibitem{MR2726097}
Joseph Bernstein and Andre Reznikov.
\newblock Subconvexity bounds for triple {$L$}-functions and representation
  theory.
\newblock {\em Ann. of Math. (2)}, 172(3):1679--1718, 2010.

\bibitem{MR2437682}
Valentin Blomer and Gergely Harcos.
\newblock The spectral decomposition of shifted convolution sums.
\newblock {\em Duke Math. J.}, 144(2):321--339, 2008.

\bibitem{blomer-harcos-milicevic}
Valentin Blomer, Gergely Harcos, and Djordje Mili\'cevi\'c.
\newblock Bounds for eigenforms on arithmetic hyperbolic 3-manifolds.
\newblock {\em Duke Math. J.}, 165(4):625--659, 2016.

\bibitem{blomer-harcos-milicevic-maga}
Valentin Blomer, Gergely Harcos, Djordje Mili\'cevi\'c, and Peter Maga.
\newblock The sup-norm problem for {GL}(2) over number fields.
\newblock arXiv:1605.09360, 2016.

\bibitem{blomer-holowinsky}
Valentin Blomer and Roman Holowinsky.
\newblock Bounding sup-norms of cusp forms of large level.
\newblock {\em Invent. Math.}, 179(3):645--681, 2010.

\bibitem{MR882297}
Colin~J. Bushnell.
\newblock Hereditary orders, {G}auss sums and supercuspidal representations of
  {${\rm GL}_N$}.
\newblock {\em J. Reine Angew. Math.}, 375/376:184--210, 1987.

\bibitem{BH06}
Colin~J. Bushnell and Guy Henniart.
\newblock {\em The Local Langlands Conjecture for $\rm GL(2)$}.
\newblock Springer-Verlag, Berlin, 2006.

\bibitem{MR533066}
Stephen Gelbart and Herv{\'e} Jacquet.
\newblock A relation between automorphic representations of {${\rm GL}(2)$} and
  {${\rm GL}(3)$}.
\newblock {\em Ann. Sci. \'Ecole Norm. Sup. (4)}, 11(4):471--542, 1978.

\bibitem{GrossPrasad:91a}
Benedict~H. Gross and Dipendra Prasad.
\newblock Test vectors for linear forms.
\newblock {\em Math. Ann.}, 291(2):343--355, 1991.

\bibitem{harcos-templier-1}
Gergely Harcos and Nicolas Templier.
\newblock On the sup-norm of {M}aass cusp forms of large level: {II}.
\newblock {\em Int. Math. Res. Not. IMRN}, 2012(20):4764--4774, 2012.

\bibitem{harcos-templier-2}
Gergely Harcos and Nicolas Templier.
\newblock On the sup-norm of {M}aass cusp forms of large level. {III}.
\newblock {\em Math. Ann.}, 356(1):209--216, 2013.

\bibitem{HL94}
Jeffrey Hoffstein and Paul Lockhart.
\newblock Coefficients of {M}aass forms and the {S}iegel zero.
\newblock {\em Ann. of Math. (2)}, 140(1):161--181, 1994.
\newblock With an appendix by Dorian Goldfeld, Hoffstein and Daniel Lieman.

\bibitem{holowinsky-soundararajan-2008}
Roman Holowinsky and Kannan Soundararajan.
\newblock Mass equidistribution for {H}ecke eigenforms.
\newblock {\em Ann. of Math. (2)}, 172(2):1517--1528, 2010.

\bibitem{MR0492088}
Roger~E. Howe.
\newblock Some qualitative results on the representation theory of {${\rm
  Gl}_{n}$} over a {$p$}-adic field.
\newblock {\em Pacific J. Math.}, 73(2):479--538, 1977.

\bibitem{MR0492087}
Roger~E. Howe.
\newblock Tamely ramified supercuspidal representations of {${\rm Gl}_{n}$}.
\newblock {\em Pacific J. Math.}, 73(2):437--460, 1977.

\bibitem{Hu:17a}
Yueke Hu.
\newblock Triple product formula and mass equidistribution on modular curves of
  level n.
\newblock {\em Int Math Res Notices}, 2017.

\bibitem{MR2449948}
Atsushi Ichino.
\newblock Trilinear forms and the central values of triple product
  {$L$}-functions.
\newblock {\em Duke Math. J.}, 145(2):281--307, 2008.

\bibitem{iwan-sar-85}
Henryk Iwaniec and Peter Sarnak.
\newblock {$L^\infty$} norms of eigenfunctions of arithmetic surfaces.
\newblock {\em Ann. of Math. (2)}, 141(2):301--320, 1995.

\bibitem{MR0401654}
Herv{\'e} Jacquet and R.~P. Langlands.
\newblock {\em Automorphic forms on {${\rm GL}(2)$}}.
\newblock Lecture Notes in Mathematics, Vol. 114. Springer-Verlag, Berlin,
  1970.

\bibitem{MR1937203}
Henry~H. Kim.
\newblock Functoriality for the exterior square of {${\rm GL}_4$} and the
  symmetric fourth of {${\rm GL}_2$}.
\newblock {\em J. Amer. Math. Soc.}, 16(1):139--183 (electronic), 2003.
\newblock With appendix 1 by Dinakar Ramakrishnan and appendix 2 by Kim and
  Peter Sarnak.

\bibitem{knapp}
A.~W. Knapp.
\newblock Local {L}anglands correspondence: the {A}rchimedean case.
\newblock In {\em Motives ({S}eattle, {WA}, 1991)}, volume~55 of {\em Proc.
  Sympos. Pure Math.}, pages 393--410. Amer. Math. Soc., Providence, RI, 1994.

\bibitem{KR14}
Andrew Knightly and Carl Ragsdale.
\newblock Matrix coefficients of depth-zero supercuspidal representations of
  {GL}(2).
\newblock {\em Involve. A Journal of Mathematics}, 7(5):669--690, 2014.

\bibitem{kutzko1}
P.~C. Kutzko.
\newblock On the supercuspidal representations of {${\rm Gl}_{2}$}.
\newblock {\em Amer. J. Math.}, 100(1):43--60, 1978.

\bibitem{MR0507254}
P.~C. Kutzko.
\newblock On the supercuspidal representations of {${\rm Gl}_{2}$}. {II}.
\newblock {\em Amer. J. Math.}, 100(4):705--716, 1978.

\bibitem{marsh15}
Simon Marshall.
\newblock Local bounds for ${L}^p$ norms of {M}aass forms in the level aspect.
\newblock {\em Preprint}, 2015.

\bibitem{michel-2009}
Philippe Michel and Akshay Venkatesh.
\newblock The subconvexity problem for {${\rm GL}_2$}.
\newblock {\em Publ. Math. Inst. Hautes \'Etudes Sci.}, (111):171--271, 2010.

\bibitem{MR987030}
Allen Moy.
\newblock A conjecture on minimal {$K$}-types for {${\rm GL}_n$} over a
  {$p$}-adic field.
\newblock In {\em Representation theory and number theory in connection with
  the local {L}anglands conjecture ({A}ugsburg, 1985)}, volume~86 of {\em
  Contemp. Math.}, pages 249--254. Amer. Math. Soc., Providence, RI, 1989.

\bibitem{moy-prasad}
Allen Moy and Gopal Prasad.
\newblock Unrefined minimal {$K$}-types for {$p$}-adic groups.
\newblock {\em Invent. Math.}, 116(1-3):393--408, 1994.

\bibitem{PDN-HQUE-LEVEL}
Paul~D. Nelson.
\newblock Equidistribution of cusp forms in the level aspect.
\newblock {\em Duke Math. J.}, 160(3):467--501, 2011.

\bibitem{nelson-padic-que}
Paul~{D}. Nelson.
\newblock Microlocal lifts and quantum unique ergodicity on {${\rm GL}\sb
  2(\mathbb{Q}_p)$}.
\newblock preprint, 2016.

\bibitem{NPS}
Paul~D. Nelson, Ameya Pitale, and Abhishek Saha.
\newblock Bounds for {R}ankin-{S}elberg integrals and quantum unique ergodicity
  for powerful levels.
\newblock {\em J. Amer. Math. Soc.}, 27(1):147--191, 2014.

\bibitem{Prasad:90a}
Dipendra Prasad.
\newblock Trilinear forms for representations of $\rm {GL}(2)$ and local
  $\epsilon-$factors.
\newblock {\em Compositio Math.}, 75(1):1--46, 1990.

\bibitem{MR1266075}
Ze{\'e}v Rudnick and Peter Sarnak.
\newblock The behaviour of eigenstates of arithmetic hyperbolic manifolds.
\newblock {\em Comm. Math. Phys.}, 161(1):195--213, 1994.

\bibitem{sahasupwhittaker}
Abhishek Saha.
\newblock Large values of newforms on {${\rm GL}(2)$} with highly ramified
  central character.
\newblock {\em Int. Math. Res. Not. IMRN}, (13):4103--4131, 2016.

\bibitem{saha-sup-level-hybrid}
Abhishek Saha.
\newblock Hybrid sup-norm bounds for {M}aass newforms of powerful level.
\newblock {\em Algebra and Number Theory}, 11(5):1009--1045, 2017.

\bibitem{Saito1993}
Hiroshi Saito.
\newblock On {T}unnell's formula for characters of {${\rm GL}(2)$}.
\newblock {\em Compositio Math.}, 85(1):99--108, 1993.

\bibitem{sarnak-progress-que}
Peter Sarnak.
\newblock Recent {P}rogress on {Q}{U}{E}.
\newblock \url{http://www.math.princeton.edu/sarnak/SarnakQUE.pdf}, 2009.

\bibitem{ralfnewform}
Ralf Schmidt.
\newblock Some remarks on local newforms for {$\rm GL(2)$}.
\newblock {\em J. Ramanujan Math. Soc.}, 17(2):115--147, 2002.

\bibitem{templier-sup}
Nicolas Templier.
\newblock On the sup-norm of {M}aass cusp forms of large level.
\newblock {\em Selecta Math. (N.S.)}, 16(3):501--531, 2010.

\bibitem{templier-large}
Nicolas Templier.
\newblock Large values of modular forms.
\newblock {\em Camb. J. Math.}, 2(1):91--116, 2014.

\bibitem{templier-sup-2}
Nicolas Templier.
\newblock Hybrid sup-norm bounds for {H}ecke--{M}aass cusp forms.
\newblock {\em J. Eur. Math. Soc. (JEMS)}, 17(8):2069--2082, 2015.

\bibitem{tunnell78}
Jerrold~B. Tunnell.
\newblock On the local {L}anglands conjecture for {$GL(2)$}.
\newblock {\em Invent. Math.}, 46(2):179--200, 1978.

\bibitem{Tunnell83}
Jerrold~B. Tunnell.
\newblock Local {$\epsilon $}-factors and characters of {${\rm GL}(2)$}.
\newblock {\em Amer. J. Math.}, 105(6):1277--1307, 1983.

\bibitem{watson-2008}
Thomas~C. Watson.
\newblock Rankin triple products and quantum chaos.
\newblock {\em arXiv e-prints}, 2008.
\newblock \url{http://arXiv.org/abs/0810.0425}.

\bibitem{xia}
Honggang Xia.
\newblock On {$L^\infty$} norms of holomorphic cusp forms.
\newblock {\em J. Number Theory}, 124(2):325--327, 2007.

\end{thebibliography}

\end{document}